\pdfoutput=1
\RequirePackage{ifpdf}
\ifpdf % We are running pdfTeX in pdf mode
\documentclass[pdftex]{sigma}
\else
\documentclass{sigma}
\fi

\numberwithin{equation}{section}

\newtheorem{Theorem}{Theorem}[section]
\newtheorem{Corollary}[Theorem]{Corollary}
\newtheorem{Lemma}[Theorem]{Lemma}
\newtheorem{Proposition}[Theorem]{Proposition}
 { \theoremstyle{definition}
\newtheorem{Definition}[Theorem]{Definition}
\newtheorem{Example}[Theorem]{Example}
\newtheorem{Note}[Theorem]{Note}
}

\begin{document}
\allowdisplaybreaks

\newcommand{\arXivNumber}{1806.10007}

\renewcommand{\thefootnote}{}

\renewcommand{\PaperNumber}{037}

\FirstPageHeading

\ShortArticleName{An Infinite-Dimensional $\square_q$-Module Obtained from the $q$-Shuffle Algebra for Affine $\mathfrak{sl}_2$}

\ArticleName{An Infinite-Dimensional $\boldsymbol{\square_q}$-Module Obtained\\ from the $\boldsymbol{q}$-Shuffle Algebra for Affine $\boldsymbol{\mathfrak{sl}_2}$}

\Author{Sarah POST~$^\dag$ and Paul TERWILLIGER~$^\ddag$}

\AuthorNameForHeading{S.~Post and P.~Terwilliger}

\Address{$^\dag$~Department of Mathematics, University of Hawai`i at Manoa, Honolulu, HI 96822, USA}
\EmailD{\href{mailto:spost@hawaii.edu}{spost@hawaii.edu}}
\URLaddressD{\url{https://math.hawaii.edu/~sarah/}}

\Address{$^\ddag$~Department of Mathematics, University of Wisconsin, Madison, WI 53706-1388, USA}
\EmailD{\href{mailto:terwilli@math.wisc.edu}{terwilli@math.wisc.edu}}

\ArticleDates{Received August 18, 2019, in final form April 19, 2020; Published online May 04, 2020}

\Abstract{Let $\mathbb F$ denote a field, and pick a nonzero $q \in \mathbb F$ that is not a root of unity. Let $\mathbb Z_4=\mathbb Z/4 \mathbb Z$ denote the cyclic group of order~4. Define a unital associative ${\mathbb F}$-algebra $\square_q$ by generators $\lbrace x_i \rbrace_{i \in \mathbb Z_4}$ and relations
\begin{gather*}
\frac{q x_i x_{i+1}-q^{-1}x_{i+1}x_i}{q-q^{-1}} = 1,\qquad
x^3_i x_{i+2} - \lbrack 3 \rbrack_q x^2_i x_{i+2} x_i + \lbrack 3 \rbrack_q x_i x_{i+2} x^2_i -x_{i+2} x^3_i = 0,
\end{gather*}
where $\lbrack 3 \rbrack_q = \big(q^3-q^{-3}\big)/\big(q-q^{-1}\big)$. Let $V$ denote a $\square_q$-module. A vector $\xi\in V$ is called NIL whenever $x_1 \xi = 0 $ and $x_3 \xi=0$ and $\xi \not=0$. The $\square_q$-module $V$ is called NIL whenever~$V$ is generated by a NIL vector. We show that up to isomorphism there exists a unique NIL $\square_q$-module, and it is irreducible and infinite-dimensional. We describe this module from sixteen points of view. In this description an important role is played by the $q$-shuffle algebra for affine $\mathfrak{sl}_2$.}


\Classification{17B37}

\section{Introduction}\label{section1}

The algebra $\square_q$ was introduced in~\cite{pospart}. It is associative, infinite-dimensional, and noncommutative. It is defined by generators and relations. There are four generators, and it is natural to identify these with the edges of an oriented four-cycle. The relations are roughly described as follows.
Each pair of adjacent edges satisfy a $q$-Weyl relation. Each pair of opposite edges satisfy the $q$-Serre relations associated with affine $\mathfrak{sl}_2$; these have degree~3 in one variable and degree~1 in the other variable. The cyclic group $\mathbb Z_4 =\mathbb Z/4 \mathbb Z$ acts on the oriented four-cycle as a group of rotational symmetries, and this induces a ${\mathbb Z}_4$-action on $\square_q$ as a group of
automorphisms.

In the theory of quantum groups, there is an algebra
$U^+_q$ called the positive part of $U_q\big(\widehat {\mathfrak {sl}}_2\big)$.
The algebra $U^+_q$ is defined by two generators, subject to
the above $q$-Serre relations
\cite[Corollary~3.2.6]{lusztig}.
The algebras $\square_q$ and $U^+_q$ are related as follows.
In the algebra $\square_q$,
each pair of opposite edges generate a subalgebra
that is isomorphic to
$U^+_q$
\cite[Proposition~5.5]{pospart}.
 This gives two subalgebras of $\square_q$ that are
isomorphic to $U^+_q$; consider their tensor product.
There is a map from this tensor product to $\square_q$,
given by multiplication in $\square_q$.
The map is an isomorphism of vector spaces
\cite[Proposition~5.5]{pospart}. Thus the
vector space $\square_q$ is isomorphic to
$U^+_q \otimes U^+_q$.

Next we discuss how $\square_q$ is related to the
 $q$-Onsager algebra ${\mathcal O}_q$.
The algebra
 ${\mathcal O}_q$
originated in algebraic combinatorics
\cite[Lemma~5.4]{tersub3},
\cite{TwoRel}
and statistical mechanics \cite[Section~1]{bas1},
\cite[Section~2]{bas2}.
Research on ${\mathcal O}_q$ is presently active
in both areas; see
\cite{BK,
TD00,
qRacahIT,ITaug,
LS99,
madrid,
uaw,
pospart,
OnsagerLusztig,
Onsageruaw,
aw}
and
\cite{bas1,bas2,
bas6,basDef,
bas8,
basXXZ,
bas4,BK05,
bas7,
basnc,
bvu}.
The algebra ${\mathcal O}_q$ is defined by two generators, subject to the
$q$-Dolan/Grady relations
\cite[Definition~4.1]{pospart}. The
$q$-Dolan/Grady relations resemble the above $q$-Serre relations,
but are slightly more complicated.
The algebras $\square_q$ and ${\mathcal O}_q$ are related as follows.
By \cite[Proposition~5.6]{pospart} there exists an
injective algebra homomorphism
${\mathcal O}_q \to \square_q$ that sends
 one ${\mathcal O}_q$-generator
to a linear combination of two adjacent $\square_q$-generators,
and the other
 ${\mathcal O}_q$-generator
to a linear combination of the remaining two $\square_q$-generators.
The only constraint on the four coefficients is that the first two
are reciprocals and the last two are reciprocals.
We just explained how $\square_q$ and ${\mathcal O}_q$ are related.
This relationship is our primary motivation
for investigating $\square_q$.

The finite-dimensional $\square_q$-modules
are investigated in \cite{yy2, yy1}. By \cite[Proposition~5.2]{yy1},
each $\square_q$-generator is invertible
on every nonzero finite-dimensional $\square_q$-module. This result gets used in \cite[Sections~8 and~9]{yy1} to obtain some remarkable $q$-exponential formulas.
In~\cite{yy2}, the finite-dimensional irreducible $\square_q$-modules
are classified up to isomorphism.
This classification is summarized as follows.
There is a family of
finite-dimensional irreducible $\square_q$-modules,
said to have type 1
\cite[Definition~6.8]{yy2}.
Any finite-dimensional irreducible $\square_q$-module can be normalized
to have type~1, by twisting it via an appropriate automorphism of $\square_q$
\cite[Note~6.9]{yy2}.
Let~$V$ denote a
finite-dimensional irreducible $\square_q$-module of type~1.
In \cite[Definition~8.6]{yy2} $V$
gets attached to a polynomial $P_V$ in one variable $z$
that has constant coefficient 1; this
$P_V$ is called the Drinfel'd polynomial of $V$.
By \cite[Proposition~1.4]{yy2}, the map
$V\mapsto P_V$
induces a bijection between the following two sets:
(i) the isomorphism classes of finite-dimensional irreducible
$\square_q$-modules of type 1;
 (ii) the polynomials in the variable $z$ that have constant
coefficient~1 and do not vanish at $z=1$.

In the present paper, our topic is a set of infinite-dimensional
 $\square_q$-modules, said to be NIL.
A~NIL $\square_q$-module is
generated by a vector that is sent to zero by a~pair of opposite $\square_q$-generators. We show that up to isomorphism
there exists a unique NIL
$\square_q$-module, and it is irreducible and infinite-dimensional.
We then describe this module from 16 points of view.
In this description an important role is played by
the $q$-shuffle algebra for affine~$\mathfrak{sl}_2$. This
algebra was introduced by
Rosso~\cite{rosso1} and described further by Green~\cite{green}.

We now summarize our results in more detail.
Let $\mathbb F$ denote a field.
All vector spaces discussed in this paper are over~$\mathbb F$.
All algebras discussed in this paper are associative, over $\mathbb F$, and
have a multiplicative identity.
Fix a nonzero $q \in \mathbb F$
that is not a root of unity.
Define the algebra
$\square_q$ by generators $\lbrace x_i \rbrace_{i \in \mathbb Z_4}$
and relations
\begin{gather}
\frac{q x_i x_{i+1}-q^{-1}x_{i+1}x_i}{q-q^{-1}} = 1,\label{eq:qweyl}\\
x^3_i x_{i+2} - \lbrack 3 \rbrack_q x^2_i x_{i+2} x_i +
\lbrack 3 \rbrack_q x_i x_{i+2} x^2_i -x_{i+2} x^3_i = 0,\label{eq:qseRRe}
\end{gather}
where $\lbrack 3 \rbrack_q = \big(q^3-q^{-3}\big)/\big(q-q^{-1}\big)$.
The relations~(\ref{eq:qweyl}) and~(\ref{eq:qseRRe}) are called the
$q$-Weyl
and
$q$-Serre
relations, respectively.
Let $V$ denote a $\square_q$-module. A vector $\xi\in V$ is called
 NIL whenever
 $x_1 \xi = 0 $ and $x_3 \xi=0$ and
 $\xi \not=0$.
The $\square_q$-module $V$ is called NIL whenever $V$ is generated by
a~NIL vector.

\looseness=-1 In this paper we obtain the following results.
Up to isomorphism, there exists a unique NIL $\square_q$-module,
which we denote by $\bf U$. The
$\square_q$-module ${\bf U}$ is irreducible, and isomorphic to $U^+_q$ as a vector space.
Recall the natural numbers $\mathbb N = \lbrace 0,1,2,\ldots\rbrace$.
The $\square_q$-module ${\bf U}$
has a unique sequence of subspaces
$\lbrace {\bf U}_n \rbrace_{n \in \mathbb N}$
such that (i)~${\bf U}_0 \not=0$;
(ii)~the sum
 ${\bf U} = \sum\limits_{n \in \mathbb N} {\bf U}_n$ is direct;
(iii)~for
$n \in \mathbb N$,
\begin{gather*}
x_0 {\bf U}_n \subseteq {\bf U}_{n+1}, \qquad
x_1 {\bf U}_n \subseteq {\bf U}_{n-1}, \qquad
x_2 {\bf U}_n \subseteq {\bf U}_{n+1}, \qquad
x_3 {\bf U}_n \subseteq {\bf U}_{n-1},
\end{gather*}
where ${\bf U}_{-1}=0$.
The sequence
$\lbrace {\bf U}_n \rbrace_{n \in \mathbb N}$ is described as follows.
The subspace ${\bf U}_0$ has dimension~1.
The nonzero vectors in~${\bf U}_0$
are precisely the NIL vectors in~$\bf U$, and each of these vectors generates~$\bf U$.
Let $\xi$ denote a NIL vector in ${\bf U}$. Then
for $n \in \mathbb N$, the subspace
 ${\bf U}_n$ is spanned by
the vectors $u_1u_2\cdots u_n \xi$ such that
$u_i \in \lbrace x_0, x_2\rbrace$ for $1 \leq i \leq n$.

We will state some more results after a few comments.
Let $\mathbb V$ denote the free algebra on two generators~$A$,~$B$.
For $n \in \mathbb N$, a word of length $n$ in $\mathbb V$
is a product $v_1 v_2 \cdots v_n$
such that $v_i \in \lbrace A, B\rbrace$
for $1 \leq i \leq n$. We interpret the word of length zero
 to be the multiplicative
identity of $\mathbb V$; this word is called trivial and denoted by $1$.
The standard basis for $\mathbb V$ consists of the words.
There exists a symmetric bilinear form
$(\, ,\,)\colon \mathbb V\times \mathbb V \to
\mathbb F$ with respect to which the standard basis is
orthonormal. The algebra ${\rm End}(\mathbb V)$ consists of the
$\mathbb F$-linear maps $\mathbb V\to \mathbb V$ with
the following property: the matrix that represents the map
with respect to the standard basis for~$\mathbb V$ has
 finitely many nonzero entries in each row.
We define an invertible $K
\in {\rm End}(\mathbb V)$ as follows.
The map $K$ is the automorphism of
the free algebra $\mathbb V$ that sends $A\mapsto q^2 A$ and
$B\mapsto q^{-2}B$. For a word $v=v_1v_2\cdots v_n$ in $\mathbb V$,
\begin{gather*}
K(v) = v
q^{\langle v_1, A\rangle+
\langle v_2, A\rangle+ \cdots+
\langle v_n, A\rangle
},
 \qquad
K^{-1}(v) = v
q^{\langle v_1, B\rangle+
\langle v_2, B\rangle+ \cdots+
\langle v_n, B\rangle
}
\end{gather*}
where
\[
\begin{array}{c|cc}
 \langle\,,\,\rangle & A & B
 \\ \hline
 A & 2 & -2 \\
B & -2 & 2
 \end{array}
\]
We define four maps in ${\rm End}(\mathbb V)$, denoted
\begin{gather}
A_L, \quad
B_L, \quad
A_R, \quad
B_R.\label{eq:ABLRIntro}
\end{gather}
For $v \in \mathbb V$,
\begin{gather*}
A_L(v) = Av, \qquad
B_L(v) = Bv, \qquad
A_R(v) = vA, \qquad
B_R(v) = vB.
\end{gather*}
We have been discussing the free algebra $\mathbb V$. There is
another algebra structure on $\mathbb V$, called the $q$-shuffle algebra
\cite{green,rosso1, rosso}. We will follow the approach of \cite{green}, which is
well suited to our purpose.
The $q$-shuffle product will be denoted by
$\star$.
 In the main body of the paper
we will describe this product in detail, and for now just make a few points.
We have $1 \star v = v \star 1 = v$ for $v \in \mathbb V$.
 For $X \in \lbrace A, B\rbrace$ and
 a nontrivial word $v= v_1v_2\cdots v_n$ in $\mathbb V$,
\begin{gather*}
X \star v =
\sum_{i=0}^n v_1 \cdots v_{i} X v_{i+1} \cdots v_n
q^{\langle v_1, X\rangle+
\langle v_2, X\rangle+
\cdots + \langle v_{i}, X\rangle},
\\
v \star X = \sum_{i=0}^n v_1 \cdots v_{i} X v_{i+1} \cdots v_n
q^{\langle v_{n},X\rangle
+\langle v_{n-1},X\rangle
+\cdots+
\langle v_{i+1},X\rangle}.
\end{gather*}
It turns out that $K$ is an automorphism of the $q$-shuffle algebra $\mathbb V$.
We define four maps in
${\rm End}(\mathbb V)$, denoted
\begin{gather}
A_{\ell}, \quad B_{\ell}, \quad A_r, \quad B_r.\label{eq:ablrIntro}
\end{gather}
For $v \in \mathbb V$,
\begin{gather*}
A_{\ell}(v) = A\star v, \qquad
B_{\ell}(v) = B\star v,
\qquad
A_r(v) = v\star A,
\qquad
B_r(v) = v\star B.
\end{gather*}
We recall the concept of an adjoint.
For $X \in {\rm End}(\mathbb V)$ there exists
a unique $X^* \in {\rm End}(\mathbb V)$ such that
$(Xu,v) = (u, X^*v)$ for all $u,v \in \mathbb V$.
The element $X^*$ is called the adjoint of $X$ with
respect to~$(\, ,\,)$. For example $K^*=K$.
We will consider
\begin{gather}
A^*_L, \quad
B^*_L, \quad
A^*_R, \quad
B^*_R\label{eq:ABLRsIntro}
\end{gather}
and
\begin{gather}\label{eq:alblsIntro}
A^*_\ell, \quad B^*_\ell, \quad A^*_r,\quad B^*_r.
\end{gather}
We acknowledge that the
maps~(\ref{eq:ABLRIntro}),
(\ref{eq:ablrIntro})
and
(\ref{eq:ABLRsIntro}),
(\ref{eq:alblsIntro})
are well known in the
literature on quantum groups and $q$-shuffle algebras.
For instance,
in \cite[Section~3.4]{kashiwara}
the maps
\begin{gather*}
e'_0= K^{-1}A^*_r,
\qquad e'_1=KB^*_r,
\qquad \quad
e''_0=A^*_\ell, \qquad
e''_1=B^*_\ell
\end{gather*}
give the Kashiwara operators for the negative
part of $U_q\big(\widehat{\mathfrak{sl}}_2\big)$.
In \cite[Lemma~3.4.2]{kashiwara} the above maps $e'_0$, $e'_1$ and $f_0=A_L$, $f_1=B_L$ are used
to obtain a module for the reduced $q$-analog~$\mathcal B_q\big(\widehat{\mathfrak{sl}}_2\big)$.
The maps $A^*_r$, $B^*_r$ are discussed in \cite[Definition~4.2]{green}, where they are called
$\Delta_0$, $\Delta_1$. The map~$A^*_R$ (resp.~$B^*_R$) is called~$\partial_0$ (resp.~$\partial_1$) in
 \cite[Section~3.1]{green} and~${\bf e'_0}$ (resp.~${\bf e'_1}$) in \cite[p.~696]{leclerc}.
In the present paper, we will put the well known maps~(\ref{eq:ABLRIntro}),
(\ref{eq:ablrIntro})
and~(\ref{eq:ABLRsIntro}), (\ref{eq:alblsIntro}) to a new use.

Let $J$ denote the 2-sided ideal of the free algebra $\mathbb V$
 generated by
\begin{gather*}
J^+=A^3 B - \lbrack 3 \rbrack_q A^2 B A+ \lbrack 3 \rbrack_q A B A^2 - B A^3,\\
J^-= B^3 A - \lbrack 3 \rbrack_q B^2 A B+ \lbrack 3 \rbrack_q B A B^2 - A B^3.
\end{gather*}
As we will see, the ideal $J$ is invariant under $K^{\pm 1}$ and
(\ref{eq:ABLRIntro}),
(\ref{eq:alblsIntro}).
The quotient algebra $\mathbb V/J$ is often denoted by $U^+_q$ and
called the positive part of
$U_q\big(\widehat {\mathfrak {sl}}_2\big)$; see for example
\cite[p.~40]{hongkang} or \cite[Corollary~3.2.6]{lusztig}.
Let $U$ denote the subalgebra of the $q$-shuffle algebra
$\mathbb V$ generated by~$A$,~$B$. As we will see, the algebra
$U$ is invariant under
$K^{\pm 1}$ and
(\ref{eq:ablrIntro}),
(\ref{eq:ABLRsIntro}).
It is well known that the algebra $U$
is isomorphic to $U^+_q$;
see \cite[Theorem~15]{rosso} or
\cite[p.~696]{leclerc}.

We are now ready to state some more results.
For notational convenience define $Q=1-q^2$.

\begin{Theorem}\label{prop:Fourmod2}For each row in the tables below,
the vector space $\mathbb V/J$ becomes a
 $\square_q$-module on which the generators
$\lbrace x_i\rbrace_{i \in \mathbb Z_4}$ act as indicated.

\centerline{
\begin{tabular}[t]{c|cccc}
{\rm module label}& $x_0$ & $x_1$ & $x_2$ & $x_3$
 \\ \hline
{\rm I}
& $A_L$
& $Q(A^*_\ell-B^*_r K)$
& $B_L$
& $Q\big(B^*_\ell-A^*_rK^{-1}\big)$\tsep{2pt}
\\
{\rm IS}
& $A_R$
& $Q(A^*_r-B^*_\ell K)$
& $B_R$
& $Q\big(B^*_r-A^*_\ell K^{-1}\big)$
\\
{\rm IT}
& $B_L$
& $Q\big(B^*_\ell-A^*_rK^{-1}\big)$
& $A_L$
& $Q(A^*_\ell-B^*_r K)$
\\
{\rm IST}
& $B_R$
& $Q\big(B^*_r-A^*_\ell K^{-1}\big)$
& $A_R$
& $Q(A^*_r-B^*_\ell K)$
 \end{tabular}
}

\medskip

\centerline{
\begin{tabular}[t]{c|cccc}
{\rm module label}& $x_0$ & $x_1$ & $x_2$ & $x_3$
 \\ \hline
{\rm II}
 & $Q(A_L- K B_R)$
& $A^*_\ell$
 & $Q\big(B_L-K^{-1}A_R\big)$
 & $B^*_\ell$\tsep{2pt}
\\
{\rm IIS}
& $Q(A_R-KB_L)$
& $A^*_r$
& $Q\big(B_R-K^{-1}A_L\big)$
& $B^*_r$
\\
{\rm IIT}
 & $Q\big(B_L-K^{-1}A_R\big)$
 & $B^*_\ell$
 & $Q(A_L- K B_R)$
& $A^*_\ell$
\\
{\rm IIST}
& $Q\big(B_R-K^{-1}A_L\big)$
& $B^*_r$
& $Q(A_R-KB_L)$
& $A^*_r$
\end{tabular}
}

Each $\square_q$-module in the tables is isomorphic to $\bf U$.
\end{Theorem}

\begin{Theorem}\label{prop:Fourmod}
For each row in the tables below,
the vector space $U$ becomes a
 $\square_q$-module on which the generators
$\lbrace x_i\rbrace_{i \in \mathbb Z_4}$ act as indicated.

\centerline{
\begin{tabular}[t]{c|cccc}
{\rm module label}& $x_0$ & $x_1$ & $x_2$ & $x_3$
 \\ \hline
{\rm III} &
$A_\ell$ &
$Q(A^*_L-B^*_RK)$ &
$B_\ell$
&$Q\big(B^*_L- A^*_R K^{-1}\big)$\tsep{2pt}
\\
{\rm IIIS} &
$A_r$ &
$Q(A^*_R-B^*_L K)$ &
$B_r$
&
$Q\big(B^*_R-A^*_L K^{-1}\big)$
\\
{\rm IIIT} &
$B_\ell$
&$Q\big(B^*_L- A^*_R K^{-1}\big)$
&
$A_\ell$ &
$Q(A^*_L-B^*_RK)$
\\
{\rm IIIST} &
$B_r$
& $Q\big(B^*_R-A^*_L K^{-1}\big)$
&
$A_r$ &
$Q(A^*_R-B^*_L K)$
 \end{tabular}
}
\medskip

\centerline{
\begin{tabular}[t]{c|cccc}
{\rm module label}& $x_0$ & $x_1$ & $x_2$ & $x_3$
 \\ \hline
{\rm IV} &
$Q(A_\ell-K B_r)$
&
$A^*_L$
&
$Q\big(B_\ell-K^{-1} A_r\big)$ &
$B^*_L$\tsep{2pt}
\\
{\rm IVS}
&
$Q(A_r-KB_\ell)$
&
$A^*_R$
&
$Q\big(B_r-K^{-1} A_\ell\big)$
&
$B^*_R$
\\
{\rm IVT} &
$Q\big(B_\ell-K^{-1} A_r\big)$ &
$B^*_L$ &
$Q(A_\ell-K B_r)$
&
$A^*_L$
\\
{\rm IVST}
&
$Q\big(B_r-K^{-1} A_\ell\big)$
&
$B^*_R$
&
$Q(A_r-KB_\ell)$
&
$A^*_R$
 \end{tabular}
}

Each $\square_q$-module in the tables is isomorphic to $\bf U$.
\end{Theorem}

We will state some additional results shortly.

We recall the concept of a derivation.
Let $\mathcal A$ denote an algebra,
and let $\varphi$, $\phi$ denote automorphisms of $\mathcal A$.
By a $(\varphi, \phi)$-derivation of $\mathcal A$ we mean an
$\mathbb F$-linear map $\delta\colon \mathcal A \to \mathcal A$
such that $\delta(uv)= \varphi(u)\delta(v)+ \delta(u) \phi(v)$ for all $u,v \in \mathcal A$.

The derivation concept is well known in the theory of quantum groups
and $q$-shuffle algebras. For instance, by \cite[Theorem~3.2]{green} the maps~$A^*_R$ and $B^*_R$ act on the $q$-shuffle algebra~$\mathbb V$ as a $(I,K)$-derivation and
$\big(I,K^{-1}\big)$-derivation, respectively.
 By \cite[Section~4.2]{green} the maps
$A^*_r$ and $B^*_r$ act on the free algebra $\mathbb V$ as
a $(I,K)$-derivation and
$\big(I,K^{-1}\big)$-derivation, respectively. Concerning
$\square_q$ we have the following results.

\begin{Theorem} \label{lem:QuotdderIntro}
For each $\square_q$-module in Theorem~{\rm \ref{prop:Fourmod2}},
 the elements $x_1$ and $x_3$ act
on the algebra~$\mathbb V/J$ as a derivation of the following sort:

\centerline{
\begin{tabular}[t]{c|cc}
{\rm module label}& $x_1$ & $x_3$
 \\ \hline
{\rm I, II} &
{\rm $(K,I)$-derivation}
&
{\rm $\big(K^{-1},I\big)$-derivation}\tsep{2pt}
\\
{\rm IS, IIS} &
{\rm $(I,K)$-derivation}
&
{\rm $(I, K^{-1})$-derivation}
\\
{\rm IT, IIT} &
{\rm $\big(K^{-1},I\big)$-derivation}
&
{\rm $( K,I)$-derivation}
\\
{\rm IST, IIST} &
{\rm $\big(I,K^{-1}\big)$-derivation}
&
{\rm $(I,K)$-derivation}
 \end{tabular}
}
\end{Theorem}

\begin{Theorem} \label{lem:Quotdder2Intro}
For each $\square_q$-module in
Theorem~{\rm \ref{prop:Fourmod}},
the elements $x_1$ and $x_3$ act
on the algebra~$U$ as
a derivation of the following sort:

\centerline{
\begin{tabular}[t]{c|cc}
{\rm module label}& $x_1$ & $x_3$
 \\ \hline
{\rm III, IV} &
{\rm $(K,I)$-derivation}
&
{\rm $\big(K^{-1},I\big)$-derivation}\tsep{2pt}
\\
{\rm IIIS, IVS} &
{\rm $(I,K)$-derivation}
&
{\rm $(I, K^{-1})$-derivation}
\\
{\rm IIIT, IVT} &
{\rm $\big(K^{-1},I\big)$-derivation}
&
{\rm $( K,I)$-derivation}
\\
{\rm IIIST, IVST} &
{\rm $\big(I,K^{-1}\big)$-derivation}
&
{\rm $(I,K)$-derivation}
 \end{tabular}
}
\end{Theorem}

 This paper is organized as follows.
Section~\ref{section2} contains some preliminaries. In Section~\ref{section3}
we recall the free algebra $\mathbb V$ on two generators.
In Section~\ref{section4}, we endow $\mathbb V$ with a bilinear form
and discuss the corresponding adjoint map.
In Section~\ref{section5}, we describe two automorphisms and
one antiautomorphism of $\mathbb V$ that will play a role in our main results.
In Section~\ref{section6}, we recall the $q$-shuffle product on $\mathbb V$, and
embed $U^+_q$ into the $q$-shuffle algebra $\mathbb V$.
 In Sections~\ref{section7}--\ref{section9}, we give a detailed description of
how the free algebra $\mathbb V$ is related to the
$q$-shuffle algebra $\mathbb V$.
In Section~\ref{section10} we introduce an algebra $\square^\vee_q$ that
is a homomorphic preimage of~$\square_q$. In Section~\ref{section11}
we give sixteen
$\square^\vee_q$-module structures on $\mathbb V$.
In Section~\ref{section12} we describe many homomorphisms between our
sixteen $\square^\vee_q$-modules.
In Section 13 we use our sixteen $\square^\vee_q$-modules to obtain
sixteen $\square_q$-modules.
In Section~\ref{section14} we show that these sixteen $\square_q$-modules are mutually
isomorphic and irreducible.
In Section~\ref{section15} we characterize these $\square_q$-modules
using the notion of a NIL $\square_q$-module.
Appendix~\ref{appendixA} contains some data on the $q$-shuffle product.
Appendix~\ref{appendixB} gives some matrix representations of the maps used
in our main results.

\section{Preliminaries}\label{section2}
We now begin our formal argument. Recall the natural numbers $\mathbb N = \lbrace 0,1,2,\ldots\rbrace$
and integers $\mathbb Z = \lbrace 0,\pm 1, \pm 2,\ldots\rbrace$.
Let $\mathbb F$ denote a~field. We will be discussing vector spaces, algebras, and tensor products.
Every vector space discussed is over~$\mathbb F$.
Every algebra discussed is over $\mathbb F$, associative,
and has a multiplicative identity. Every tensor product discussed is
over~$\mathbb F$.
Let $\mathcal A$
denote an algebra.
By an {\it automorphism} (resp.
{\it antiautomorphism})
of $\mathcal A$ we
mean an $\mathbb F$-linear bijection
$\gamma\colon \mathcal A \to \mathcal A$ such that
$\gamma(uv) = \gamma(u)\gamma(v)$
(resp.\
$\gamma(uv) = \gamma(v)\gamma(u)$)
for all $u,v \in \mathcal A$.
Let~$\varphi$,~$\phi$ denote automorphisms of $\mathcal A$.
By a {\it $(\varphi, \phi)$-derivation} of $\mathcal A$ we mean an
$\mathbb F$-linear map $\delta\colon \mathcal A \to \mathcal A$
such that $\delta(uv)= \varphi(u)\delta(v)+ \delta(u) \phi(v)$ for all $u,v \in \mathcal A$.
The set of all $(\varphi,\phi)$-derivations of~$\mathcal A$ is closed under
addition and scalar multiplication, and is therefore a vector space
over~$\mathbb F$.
Let~$\delta$ denote a $(\varphi,\phi)$-derivation of~$\mathcal A$.
Then $\delta(1)=0$.
For an automorphism $\sigma$ of~$\mathcal A$, the composition
$\delta\sigma$ is a $(\varphi \sigma, \phi \sigma)$-derivation
of~$\mathcal A$, and $\sigma \delta$
is a $(\sigma \varphi,\sigma \phi)$-derivation
of~$\mathcal A$.
Let $J$ denote a 2-sided ideal of $\mathcal A$ with $J\not=\mathcal A$,
and consider the quotient algebra $\overline {\mathcal A} = \mathcal A /J$.
Assume that
$\varphi(J)=J$ and $\phi(J)=J$.
Then there exist automorphisms $\overline \varphi $ and $\overline \phi$ of
$\overline {\mathcal A}$ such that $\overline \varphi (a+J) =
\varphi(a)+ J$ and
 $\overline \phi (a+J) =
\phi(a)+ J$
for all $a \in \mathcal A$.
Assume further that
$\delta(J) \subseteq J$. Then there exists a
 $\big(\overline \varphi, \overline \phi\big)$-derivation $\overline \delta$ of
$\overline {\mathcal A}$ such that $\overline \delta (a+J) = \delta(a)+J$ for
all $a \in \mathcal A$. In the main body of the paper we will suppress
the overline notation.

Fix a nonzero $q \in \mathbb F$ that is not a root of unity.
Recall the notation
\begin{gather*}
\lbrack n \rbrack_q = \frac{q^n-q^{-n}}{q-q^{-1}},
\qquad n \in \mathbb Z.
\end{gather*}

\section[The free algebra $\mathbb V$ with generators $A$, $B$]{The free algebra $\boldsymbol{\mathbb V}$ with generators $\boldsymbol{A}$, $\boldsymbol{B}$}\label{section3}

Let $A$, $B$ denote noncommuting indeterminates, and let~$\mathbb V$ denote the free algebra
with genera\-tors~$A$,~$B$. For $n \in \mathbb N$, a {\it word of length $n$} in $\mathbb V$
is a product $v_1 v_2 \cdots v_n$ such that $v_i \in \lbrace A, B\rbrace$
for $1 \leq i \leq n$. We interpret the word of length zero
 to be the multiplicative identity in $\mathbb V$; this word is called {\it trivial} and denoted by~$1$.
The vector space $\mathbb V$ has a basis consisting of its words; this basis is called {\it standard}.

\begin{Definition}\label{def:grading}For $n \in \mathbb N$, let ${\mathbb V}_n$ denote the
subspace of $\mathbb V$ spanned by the words of length $n$. For notational convenience define ${\mathbb V}_{-1}=0$.
\end{Definition}

Referring to Definition~\ref{def:grading}, the dimension of ${\mathbb V}_n$ is $2^n$. We have
\begin{gather}\label{eq:grading}
\mathbb V=\sum_{n \in \mathbb N} {\mathbb V}_n \qquad \mbox{\rm (direct sum).}
\end{gather}
The vector $1$ is a basis for ${\mathbb V}_0$.
For $r,s\in \mathbb N$ we have ${\mathbb V}_r {\mathbb V}_s \subseteq
{\mathbb V}_{r+s}$.
By these comments
the sum~(\ref{eq:grading}) is a grading of $\mathbb V$ in the sense of
\cite[p.~704]{rotman}.

Let ${\rm End}(\mathbb V)$ denote the algebra consisting of the
$\mathbb F$-linear maps $\mathbb V \to \mathbb V$
with the following property: the matrix that represents the map
with respect to the standard basis for $\mathbb V$ has
 finitely many nonzero entries in each row.

\begin{Definition}\label{def:ALetc}
We define four maps in ${\rm End}(\mathbb V)$, denoted
\begin{gather}
A_L, \quad
B_L, \quad
A_R, \quad
B_R.
\label{eq:leftmult}
\end{gather}
For $v \in \mathbb V$,
\begin{gather*}
A_L(v)=Av, \qquad
B_L(v)=Bv, \qquad
A_R(v)=vA, \qquad
B_R(v)=vB.%\label{eq:leftrightmult}
\end{gather*}
\end{Definition}
\begin{Lemma}\label{lem:ABraise}
For $n \in \mathbb N$,
\begin{gather*}
A_L {\mathbb V}_n \subseteq {\mathbb V}_{n+1}, \qquad
B_L {\mathbb V}_n \subseteq {\mathbb V}_{n+1}, \qquad
A_R {\mathbb V}_n \subseteq {\mathbb V}_{n+1}, \qquad
B_R {\mathbb V}_n \subseteq {\mathbb V}_{n+1}.
\end{gather*}
\end{Lemma}

The following lemma is about $A_L$ and $B_L$; a similar result holds for~$A_R$ and~$B_R$.

\begin{Lemma}\label{lem:WAB} Let $W$ denote a subspace of $\mathbb V$ that is closed
under $A_L$ and $B_L$. Assume that $1 \in W$. Then $W=\mathbb V$.
\end{Lemma}
\begin{proof} The free algebra $\mathbb V$ is generated by~$A$,~$B$.
\end{proof}

\begin{Definition}\label{def:idealJ}Let $J$ denote the 2-sided ideal of the free algebra $\mathbb V$ generated by
\begin{gather}
J^+=A^3 B - \lbrack 3 \rbrack_q A^2 B A+ \lbrack 3 \rbrack_q A B A^2 - B A^3,\label{eq:jplus}\\
J^-= B^3 A - \lbrack 3 \rbrack_q B^2 A B+ \lbrack 3 \rbrack_q B A B^2 - A B^3.\label{eq:jminus}
\end{gather}
\end{Definition}

\begin{Definition}\label{def:pp}The quotient algebra $\mathbb V/J$ is often denoted by~$U^+_q$ and called the {\it positive part of $U_q\big({\widehat{\mathfrak{sl}}}_2\big)$}; see for example
\cite[p.~40]{hongkang} or \cite[Corollary~3.2.6]{lusztig}.
\end{Definition}

\section[A bilinear form on $\mathbb V$]{A bilinear form on $\boldsymbol{\mathbb V}$}\label{section4}

We continue to discuss the free algebra $\mathbb V$ with generators $A$, $B$. In this section we endow $\mathbb V$ with a symmetric nondegenerate bilinear form. We describe the
corresponding adjoints of the four maps listed in~(\ref{eq:leftmult}).

\begin{Definition}\label{def:bilf}\rm Define a bilinear form $(\, ,\,)\colon \mathbb V\times \mathbb V \to
\mathbb F$ as follows. Recall that the standard basis for~$\mathbb V$
consists of the words in $A$, $B$. This basis is orthonormal with respect to $(\,,\,)$.
\end{Definition}

The bilinear form $(\,,\,)$ is symmetric and nondegenerate. The summands in~(\ref{eq:grading}) are mutually orthogonal with respect to $(\,,\,)$.

We now recall the adjoint map. By linear algebra,
for $X \in {\rm End}(\mathbb V)$ there exists a unique
$X^* \in {\rm End}(\mathbb V)$ such that
$(Xu,v) = (u, X^*v)$ for all $u,v \in \mathbb V$.
With respect to the standard basis for $\mathbb V$, the matrices representing
$X$ and $X^*$ are transposes.
The element $X^*$ is called the {\it adjoint} of $X$ with
respect to
$(\, ,\,)$.
The adjoint map
${\rm End}(\mathbb V) \to
{\rm End}(\mathbb V)$,
$X \mapsto X^*$ is an antiautomorphism of
${\rm End}(\mathbb V)$.
We now consider
\begin{gather}
A^*_L, \quad
B^*_L, \quad
A^*_R, \quad
B^*_R.\label{eq:ABLR}
\end{gather}

\begin{Lemma}\label{lem:Asaction}We have
\begin{gather}
A^*_L(1)=0, \qquad
B^*_L(1)=0, \qquad
A^*_R(1)=0, \qquad
B^*_R(1)=0.\label{eq:ABaction}
\end{gather}
Moreover for $v \in \mathbb V$,
\begin{alignat*}{5}
& A^*_L(Av) = v,\qquad &&
A^*_L(Bv) = 0,\qquad &&
B^*_L(Av) = 0,\qquad &&
B^*_L(Bv) = v, &\\
 &A^*_R(vA) = v,\qquad &&
A^*_R(vB) = 0,\qquad &&
B^*_R(vA) = 0,\qquad &&
B^*_R(vB) = v. &
\end{alignat*}
\end{Lemma}
\begin{proof} This follows from
Definition~\ref{def:bilf}
and the meaning of the adjoint.
We illustrate with a~detailed proof of $A^*_L(Av)=v$.
For $u \in \mathbb V$,
\begin{gather*}
(u, A^*_L(Av)) = (A_L(u), Av) = (Au,Av)=(u,v).
\end{gather*}
Therefore $A^*_L(Av)=v$ since $(\,,\,)$ is nondegenerate.
\end{proof}

We now describe how the maps~(\ref{eq:ABLR})
act on the standard basis for $\mathbb V$.
In view of~(\ref{eq:ABaction}),
we focus on the nontrivial basis elements.
Recall that the Kronecker delta $\delta_{r,s}$ is equal to~1 if~$r=s$, and~0 if~$r\not=s$.
\begin{Lemma}
\label{lem:Astar}
For an integer $n\geq 1$ and a word
$v=v_1v_2\cdots v_n$ in $\mathbb V$,
\begin{alignat*}{3}
& A^*_L(v) =v_2\cdots v_n \delta_{v_1, A},
\qquad && B^*_L(v) =v_2\cdots v_n \delta_{v_1, B},& \\
& A^*_R(v) =v_1\cdots v_{n-1} \delta_{v_n, A},\qquad &&
B^*_R(v) = v_1\cdots v_{n-1} \delta_{v_n, B}.&
\end{alignat*}
\end{Lemma}
\begin{proof} Use Lemma~\ref{lem:Asaction}.
\end{proof}
\begin{Lemma}\label{lem:ABsLower}
For $n \in \mathbb N$,
\begin{gather*}
A^*_L {\mathbb V}_n \subseteq {\mathbb V}_{n-1}, \qquad
B^*_L {\mathbb V}_n \subseteq {\mathbb V}_{n-1}, \qquad
A^*_R {\mathbb V}_n \subseteq {\mathbb V}_{n-1}, \qquad
B^*_R {\mathbb V}_n \subseteq {\mathbb V}_{n-1}.
\end{gather*}
\end{Lemma}
\begin{proof} Use Lemma
\ref{lem:Astar}.
\end{proof}

The next two lemmas are about $A^*_L$ and $B^*_L$; similar results
hold for $A^*_R$ and $B^*_R$.
\begin{Lemma}\label{lem:ABKer}
For $v \in \mathbb V$ the following are equivalent:
\begin{enumerate}\itemsep=0pt
\item[\rm (i)] $v \in {\mathbb V}_0$;
\item[\rm (ii)] $A^*_L v=0$ and $B^*_L v = 0$.
\end{enumerate}
\end{Lemma}
\begin{proof}${\rm (i)} \Rightarrow {\rm (ii)}$: By~(\ref{eq:ABaction}).

${\rm (ii)} \Rightarrow {\rm (i)}$: Consider the orthogonal
direct sum $\mathbb V= \mathbb V_0 + A\mathbb V + B \mathbb V$.
We have $0 = (A^*_Lv, \mathbb V) = (v, A\mathbb V)$ and
$0 = (B^*_Lv, \mathbb V) = (v, B\mathbb V)$.
The vector $v$ is orthogonal to both
$A\mathbb V$ and $B\mathbb V$, and is therefore contained in $\mathbb V_0$.
\end{proof}

\begin{Lemma}\label{lem:ABirred}
Let $W$ denote a nonzero subspace of $\mathbb V$ that is closed under
$A^*_L$ and $B^*_L$. Then $1 \in W$.
\end{Lemma}
\begin{proof} There exists $0 \not=w \in W$. Write $w = \sum\limits_{i=0}^n w_i$ with $w_i \in \mathbb V_i$ for $0 \leq i \leq n$ and $w_n \not=0$. Call~$n$ the degree of~$w$. Choose $w$ such that this degree is minimal. By assumption $A^*_L w \in W$. By Lemma~\ref{lem:ABsLower} the vector
$A^*_Lw$ is either~0 or has degree less than $n$. So $A^*_Lw=0$ by the minimality of $n$.
Similarly $B^*_Lw=0$. Now $w \in \mathbb V_0$ by Lemma~\ref{lem:ABKer}. The result follows.
\end{proof}

\section{Some automorphisms and antiautomorphisms}\label{section5}

We continue to discuss the free algebra $\mathbb V$ with generators $A$, $B$. In this section we introduce
the automorphisms $K$, $T$ of $\mathbb V$ and the antiautomorphism~$S$ of~$\mathbb V$.
\begin{Definition}\label{def:Kdef} Let $K$ denote the automorphism of the free algebra $\mathbb V$
that sends $A\mapsto q^2 A$ and $B\mapsto q^{-2}B$.
\end{Definition}

The automorphism $K$ is described as follows. For a word $v=v_1 v_2 \cdots v_n$ in $\mathbb V$,
\begin{gather*}
K(v) =v q^{\langle v_1, A\rangle+
\langle v_2, A\rangle+ \cdots+
\langle v_n, A\rangle},
\qquad
K^{-1}(v) =
v q^{\langle v_1, B\rangle+
\langle v_2, B\rangle+ \cdots+
\langle v_n, B\rangle
},
\end{gather*}
where
\[
\begin{array}{c|cc}
\langle\,,\,\rangle & A & B
 \\ \hline
A &2 & -2 \\
B & -2 & 2
 \end{array}
\]

We have $K\mathbb V_n = \mathbb V_n $ for $n \in \mathbb N$, and also $K^*=K$.

\begin{Definition}\label{def:S}Define $S \in {\rm End}(\mathbb V)$ such that for each word $v = v_1 v_2 \cdots v_n$ in $\mathbb V$,
\begin{gather*}
S(v) = v_n \cdots v_2 v_1.
\end{gather*}
\end{Definition}
The map $S$ is described as follows. It is the unique antiautomorphism of the free algebra $\mathbb V$
that fixes~$A$ and~$B$. We have
$S \mathbb V_n = \mathbb V_n$ for $n \in \mathbb N$.
We have $(S(u), S(v))= (u,v) $ for $u,v \in \mathbb V$. By this and
$S^2=I$ we obtain
$(S(u),v)= (u,S(v))$ for all $u,v \in \mathbb V$. Therefore
$S^*=S$.
\begin{Lemma}\label{lem:scom} We have $KS=SK$ and
\begin{alignat}{5}
&A_L S = S A_R, \qquad &&
A_R S = S A_L, \qquad &&
B_L S = S B_R, \qquad &&
B_R S = S B_L, &\label{eq:ALS}
\\
& A^*_L S = S A^*_R, \qquad &&
A^*_R S = S A^*_L, \qquad &&
B^*_L S = S B^*_R, \qquad &&
B^*_R S = S B^*_L. &\label{eq:ALS2}
\end{alignat}
\end{Lemma}
\begin{proof}The first five equations are readily verified by applying both sides to
a word in~$\mathbb V$. We illustrate with a detailed proof of
$A_LS=SA_R$.
For a word $v$ in $\mathbb V$,
\begin{gather*}
A_L S(v) = AS(v) = S(A)S(v)= S(vA) = S A_R(v).
\end{gather*}
Therefore
$A_LS=SA_R$.
To obtain the equations~(\ref{eq:ALS2}), apply the adjoint map to each equation in~(\ref{eq:ALS}).
\end{proof}

\begin{Definition}\label{def:T}Let $T$ denote the automorphism of the free algebra $\mathbb V$ that
swaps~$A$ and~$B$.
\end{Definition}

The map $T$ is described as follows. We have $T\mathbb V_n=\mathbb V_n$ for $n \in \mathbb N$. We have
$(T(u), T(v))=(u,v)$ for $u,v \in \mathbb V$.
By this and $T^2=I$ we obtain
$(T(u), v)=(u,T(v))$ for $u,v \in \mathbb V$.
Therefore
$T^*=T$. We have
$ST=TS$.
\begin{Lemma}
\label{lem:Tact}
We have $KT=TK^{-1}$ and
\begin{alignat}{5}
& A_L T = T B_L, \qquad &&
A_R T = T B_R, \qquad &&
B_L T = T A_L, \qquad &&
B_R T = T A_R, & \label{eq:TABL}\\
& A^*_L T = T B^*_L, \qquad &&
A^*_R T = T B^*_R, \qquad &&
B^*_L T = T A^*_L, \qquad &&
B^*_R T = T A^*_R. &\label{eq:TABL2}
\end{alignat}
\end{Lemma}
\begin{proof} The first five equations are readily verified by applying both sides to a word in $\mathbb V$.
We illustrate with a detailed proof of $A_L T = T B_L$.
For a word $v$ in $\mathbb V$,
\begin{gather*}
A_L T(v) = A T(v) = T(B)T(v) = T(Bv) = T B_L(v).
\end{gather*}
Therefore $A_L T = T B_L$. To obtain the equations~(\ref{eq:TABL2}), apply the adjoint map to each equation in~(\ref{eq:TABL}).
\end{proof}

Recall $J^{\pm}$, $J$ from Definition~\ref{def:idealJ}.

\begin{Lemma}\label{lem:Jinv}We have
\begin{gather}\label{lem:KST}
K\big(J^{\pm }\big) = q^{\pm 4} J^{\pm },
\qquad
S\big(J^{\pm }\big) = -J^{\pm },
\qquad
T\big(J^{\pm }\big) = J^{\mp }.
\end{gather}
Moreover
\begin{gather}\label{eq:KSTJ}
K (J) = J, \qquad S (J) = J, \qquad T (J) = J.
\end{gather}
\end{Lemma}
\begin{proof} \looseness=1 The relations
(\ref{lem:KST}) are routinely checked using
(\ref{eq:jplus}) and
(\ref{eq:jminus}).
Line (\ref{eq:KSTJ}) follows from~(\ref{lem:KST}).
\end{proof}

\section[The $q$-shuffle algebra $\mathbb V$ and the map $\theta$]{The $\boldsymbol{q}$-shuffle algebra $\boldsymbol{\mathbb V}$ and the map $\boldsymbol{\theta}$}\label{section6}

We have been discussing the free algebra $\mathbb V$. There is another algebra structure on $\mathbb V$, called the $q$-shuffle algebra~\cite{green,rosso1,rosso}.
For this algebra the product
is denoted by $\star$. To describe this product, we start
with some special cases.
We have $1 \star v = v \star 1 = v$ for $v \in \mathbb V$.
 For $X \in \lbrace A, B\rbrace$ and a nontrivial word $v= v_1v_2\cdots v_n$ in $\mathbb V$,
\begin{gather}\label{eq:Xcv}
X \star v =\sum_{i=0}^n v_1 \cdots v_{i} X v_{i+1} \cdots v_n
q^{\langle v_1, X\rangle+
\langle v_2, X\rangle+
\cdots + \langle v_{i}, X\rangle},\\
v \star X = \sum_{i=0}^n v_1 \cdots v_{i} X v_{i+1} \cdots v_n
q^{\langle v_{n},X\rangle
+
\langle v_{n-1},X\rangle
+
\cdots
+
\langle v_{i+1},X\rangle}.\label{eq:vcX}
\end{gather}
For nontrivial words $u=u_1u_2\cdots u_r$
and $v=v_1v_2\cdots v_s$ in $\mathbb V$,
\begin{gather}\label{eq:uvcirc}
u \star v = u_1\bigl((u_2\cdots u_r) \star v\bigr)
+ v_1\bigl(u \star (v_2 \cdots v_s)\bigr)
q^{\langle u_1, v_1\rangle +
\langle u_2, v_1\rangle +
\cdots+\langle u_r, v_1\rangle},
\\
\label{eq:uvcirc2}
u\star v =\bigl(u \star (v_1 \cdots v_{s-1})\bigr)v_s +
\bigl((u_1 \cdots u_{r-1}) \star v\bigr)u_r
q^{
\langle u_r, v_1\rangle +
\langle u_r, v_2\rangle + \cdots +
\langle u_r, v_s\rangle}.
\end{gather}
For
 $r,s \in \mathbb N$ we have $\mathbb V_r \star \mathbb V_s \subseteq \mathbb V_{r+s}$.
 Therefore the sum~(\ref{eq:grading}) is a grading for the $q$-shuffle algebra $\mathbb V$.
The following examples illustrate the shuffle product.
\begin{Example}\label{ex:AB} We have
\begin{alignat*}{3}
&A \star A = \big(1+q^2\big)AA, \qquad && A \star B = AB+ q^{-2}BA,&\\
&B \star A = BA + q^{-2}AB, \qquad&& B \star B = \big(1+q^2\big)BB.&
\end{alignat*}
\end{Example}

Appendix~\ref{appendixA} contains additional examples. Using these examples
(or by \cite[p.~10]{green})
one obtains
\begin{gather}
A \star A \star A \star B -
\lbrack 3 \rbrack_q
A \star A \star B \star A +
\lbrack 3 \rbrack_q
A \star B \star A \star A -
B \star A \star A \star A = 0,
\label{eq:qsc1}\\
B \star B \star B \star A -
\lbrack 3 \rbrack_q
B \star B \star A \star B +
\lbrack 3 \rbrack_q
B \star A \star B \star B -
A \star B \star B \star B = 0.\label{eq:qsc2}
\end{gather}

For the $q$-shuffle algebra $\mathbb V$ let $U$ denote the
subalgebra generated by $A$, $B$. The algebra~$U$ is described as follows.
There exists an algebra homomorphism $\theta$ from
the free algebra~$\mathbb V$ to the $q$-shuffle algebra $\mathbb V$, that sends
$A\mapsto A$ and $B\mapsto B$. By construction $\theta(\mathbb V)=U$.
Comparing~(\ref{eq:jplus}),~(\ref{eq:jminus}) with~(\ref{eq:qsc1}),~(\ref{eq:qsc2})
 we obtain $\theta(J^\pm)=0$.
Recall that $J^\pm$ generate the 2-sided ideal $J$ of
the free algebra $\mathbb V$. Consequently $\theta (J)=0$,
so the kernel ${\rm ker}(\theta)$ contains $J$.
 By
\cite[Theorem~15]{rosso} we have
 ${\rm ker}(\theta)=J$,
so $\theta$ induces an algebra isomorphism $\mathbb V/J\to U$.
By this and $\mathbb V/J=U^+_q$, we get an algebra
isomorphism $U^+_q\to U$.
This isomorphism is discussed
around \cite[Theorem~15]{rosso} and also \cite[p.~696]{leclerc}.

Our next goal is to describe how $\theta$ acts on the
standard basis for $\mathbb V$.
By construction $\theta(1)=1$.
Pick an integer $n\geq 1$.
We view the symmetric group $S_n$
as the group of permutations of the set
$\lbrace 1,2,\ldots, n\rbrace$. For $\sigma \in S_n$, by an
{\it inversion
for $\sigma$} we mean an ordered pair $(i,j)$ of integers
such that $1 \leq i< j \leq n$ and $\sigma(i) > \sigma(j)$.
Let ${\rm Inv}(\sigma)$ denote the set of inversions for $\sigma$.

\begin{Lemma}\label{lem:thform}For an integer $n\geq 1$ and a word $v=v_1 v_2 \cdots v_n$ in $\mathbb V$,
\begin{gather*}%\label{eq:theta}
\theta(v) =\sum_{\sigma \in S_n}v_{\sigma(1)}v_{\sigma(2)} \cdots v_{\sigma(n)} \prod_{(i,j)\in {\rm Inv}(\sigma)} q^{\langle v_{\sigma(i)}, v_{\sigma(j)}\rangle}.
\end{gather*}
\end{Lemma}
\begin{proof}By induction on $n$, using (\ref{eq:Xcv}) or~(\ref{eq:vcX}).
\end{proof}

By Lemma~\ref{lem:thform} we have $\theta \mathbb V_n \subseteq \mathbb V_n $ for $n \in \mathbb N$.

\begin{Lemma}\label{lem:uvtheta}For words $u=u_1u_2\cdots u_r$ and $v=v_1v_2\cdots v_s$ in $\mathbb V$,
\begin{gather}
(\theta(u),v)= (u, \theta(v)).\label{eq:thetauv}
\end{gather}
For $r\not=s$ the common value~\eqref{eq:thetauv} is $0$. For $r=s=0$ the common value~\eqref{eq:thetauv} is~$1$. For $r=s\geq 1$ the common value~\eqref{eq:thetauv}
is equal to
\begin{gather*}
 \sum_{\sigma}\prod_{(i,j)\in {\rm Inv}(\sigma)}q^{\langle v_i, v_j \rangle},
\end{gather*}
where the sum is over all $\sigma \in S_r$ such that $v_k = u_{\sigma(k)}$ for $1 \leq k \leq r$.
\end{Lemma}
\begin{proof} By Lemma~\ref{lem:thform} and since the standard basis for $\mathbb V$ is orthonormal with respect to $(\,,\,)$.
\end{proof}

The following result is a reformulation of \cite[Theorem~4.5]{green}.
\begin{Corollary}\label{cor:sdtheta}
We have $\theta^*=\theta$.
\end{Corollary}
\begin{proof} By Lemma~\ref{lem:uvtheta} we have $(\theta(u),v)= (u, \theta(v))$ for $u,v \in \mathbb V$.
\end{proof}

Let the set $U^\perp$ consist of the vectors in $\mathbb V$ that are orthogonal to $U$ with respect to
 $(\,,\,)$.
\begin{Lemma}\label{lem:ds}We have $J=U^\perp$.
\end{Lemma}
\begin{proof} Use $J={\rm ker}(\theta)$ and
Corollary~\ref{cor:sdtheta}, along with the fact that $(\,,\,)$ is nondegenerate.
Here are the details. For $v \in \mathbb V$,
\begin{gather*}
v \in J \ \Leftrightarrow\
\theta(v)=0 \ \Leftrightarrow \
( \theta(v),\mathbb V)=0\ \Leftrightarrow \
(v, \theta(\mathbb V))=0 \ \Leftrightarrow\
(v,U)=0 \ \Leftrightarrow \ v \in U^\perp.\tag*{\qed}
\end{gather*}\renewcommand{\qed}{}
\end{proof}

Note that $\theta(U)\subseteq U$ since $U \subseteq \mathbb V$ and $\theta(\mathbb V)=U$.

\begin{Lemma}\label{lem:thetaKTS}For the $q$-shuffle algebra $\mathbb V$, the maps $K$ and $T$
are automorphisms and $S$ is an antiautomorphism.
\end{Lemma}
\begin{proof}Use the definition of the $q$-shuffle product.
\end{proof}

\begin{Lemma}\label{lem:thetaKTS2}We have
\begin{gather*}
K \theta = \theta K, \qquad
S \theta = \theta S, \qquad
T \theta = \theta T.
\end{gather*}
\end{Lemma}
\begin{proof} Apply each side to a word in $\mathbb V$, and evaluate the result using Lemma~\ref{lem:thetaKTS}.
\end{proof}

\begin{Lemma}\label{lem:KSTU}We have
\begin{gather*}
 K (U) = U, \qquad S (U) = U, \qquad T (U) = U.
\end{gather*}
\end{Lemma}
\begin{proof} By Lemma~\ref{lem:thetaKTS2} and since each of $K$, $S$, $T$ is invertible. We give the details for the equation involving~$K$. We have
\begin{gather*}
K(U) = K \theta (\mathbb V) = \theta K(\mathbb V) = \theta(\mathbb V) = U.\tag*{\qed}
\end{gather*}\renewcommand{\qed}{}
\end{proof}

\section[The maps $A_{\ell}$, $B_{\ell}$, $A_{r}$, $B_{r}$]{The maps $\boldsymbol{A_{\ell}}$, $\boldsymbol{B_{\ell}}$, $\boldsymbol{A_{r}}$, $\boldsymbol{B_{r}}$}\label{section7}

In Definition \ref{def:ALetc} we used the free algebra~$\mathbb V$ to obtain the four maps listed in~(\ref{eq:leftmult}). In this section we use the $q$-shuffle algebra $\mathbb V$ to obtain
four analogous maps. We investigate how these maps and their adjoints are related to $S$, $T$, $\theta$.

\begin{Definition}\label{def:AlEtc}We define four maps in ${\rm End}(\mathbb V)$, denoted
\begin{gather}\label{eq:alblarbr}
A_{\ell}, \quad
B_{\ell}, \quad
A_{r}, \quad
B_{r}.
\end{gather}
For $v \in \mathbb V$,
\begin{gather*}
A_{\ell}(v) = A\star v, \qquad
B_{\ell}(v) = B\star v, \qquad
A_{r}(v) = v \star A, \qquad
B_{r}(v) = v \star B.
\end{gather*}
\end{Definition}

\begin{Lemma}\label{lem:ABlRaise2} For $n \in \mathbb N$,
\begin{gather*}
A_{\ell} \mathbb V_n \subseteq \mathbb V_{n+1},
\qquad
B_{\ell} \mathbb V_n \subseteq \mathbb V_{n+1},
\qquad
A_{r} \mathbb V_n \subseteq \mathbb V_{n+1},
\qquad
B_{r} \mathbb V_n \subseteq \mathbb V_{n+1}.
\end{gather*}
\end{Lemma}
\begin{proof} By (\ref{eq:Xcv}), (\ref{eq:vcX}) and Definition~\ref{def:AlEtc}.
\end{proof}

The following lemma is about $A_\ell$ and $B_{\ell}$; a similar result holds for
 $A_r$ and $B_r$.
\begin{Lemma}\label{lem:ABirred2} Let $W$ denote a subspace of $U$ that is closed under $A_\ell$, $B_{\ell}$. Assume that $1 \in W$. Then $W=U$.
\end{Lemma}
\begin{proof} By Definition~\ref{def:AlEtc}, and since $U$ is the subalgebra of the $q$-shuffle algebra $\mathbb V$ generated by~$A$,~$B$.
\end{proof}

We now consider
\begin{gather}\label{eq:AlBl}
A^*_\ell, \quad B^*_\ell, \quad A^*_r,\quad B^*_r.
\end{gather}

\begin{Lemma}\label{lem:dualaction}For a word $v=v_1v_2\cdots v_n$ in $\mathbb V$,
\begin{gather*}
A^*_\ell(v) = \sum_{i=0}^n
 v_1 \cdots v_{i-1}
	 v_{i+1}\cdots v_n
	 \delta_{v_i,A}
	 q^{\langle v_1,A\rangle+ \langle v_2, A\rangle + \cdots +
	 \langle v_{i-1},A\rangle},
\\
B^*_\ell(v) = \sum_{i=0}^n
 v_1 \cdots v_{i-1}
	 v_{i+1}\cdots v_n
	 \delta_{v_i,B}
	 q^{\langle v_1,B\rangle+ \langle v_2, B\rangle + \cdots +
	 \langle v_{i-1},B\rangle},
\\
A^*_r(v) = \sum_{i=0}^n
 v_1 \cdots v_{i-1}
	 v_{i+1}\cdots v_n
	 \delta_{v_i,A}
	 q^{\langle v_n,A\rangle+ \langle v_{n-1}, A\rangle + \cdots +
	 \langle v_{i+1},A\rangle},
\\
B^*_r(v) = \sum_{i=0}^n
 v_1 \cdots v_{i-1}
	 v_{i+1}\cdots v_n
	 \delta_{v_i,B}
	 q^{\langle v_n,B\rangle+ \langle v_{n-1}, B\rangle + \cdots +
	 \langle v_{i+1},B\rangle}.
\end{gather*}
\end{Lemma}
\begin{proof} For each map $X$ in~(\ref{eq:alblarbr}), use~(\ref{eq:Xcv}),~(\ref{eq:vcX})
to compute the matrix representing~$X$
with respect to the standard basis for~$\mathbb V$.
The transpose of this matrix represents~$X^*$
with respect to the standard basis for $\mathbb V$.
The result follows.
\end{proof}

\begin{Lemma}\label{lem:ABLower2}
For $n \in \mathbb N$,
\begin{gather*}
A^*_{\ell} \mathbb V_n \subseteq \mathbb V_{n-1},
\qquad
B^*_{\ell} \mathbb V_n \subseteq \mathbb V_{n-1},
\qquad
A^*_{r} \mathbb V_n \subseteq \mathbb V_{n-1},
\qquad
B^*_{r} \mathbb V_n \subseteq \mathbb V_{n-1}.
\end{gather*}
\end{Lemma}
\begin{proof}Use
Lemma \ref{lem:dualaction}.
\end{proof}

We now consider how the maps (\ref{eq:alblarbr}), (\ref{eq:AlBl}) are related to $S$, $T$, $\theta$.
\begin{Lemma}\label{lem:SAB}
We have
\begin{alignat}{5}
& S A_\ell = A_r S, \qquad&&
 S A_r = A_\ell S, \qquad&&
 S B_\ell = B_r S, \qquad&&
 S B_r = B_\ell S,&\label{eq:Scom1}\\
& S A^*_\ell = A^*_r S, \qquad &&
 S A^*_r = A^*_\ell S, \qquad &&
 S B^*_\ell = B^*_r S, \qquad &&
 S B^*_r = B^*_\ell S.&\label{eq:Scom2}
\end{alignat}
\end{Lemma}
\begin{proof} The equations~(\ref{eq:Scom1}) follow from Lemma~\ref{lem:thetaKTS}. We illustrate with a detailed proof of
$S A_{\ell} = A_r S$. For $v \in \mathbb V$,
\begin{gather*}
S A_{\ell}(v) = S(A \star v) = S(v)\star S(A) = S(v) \star A = A_r S(v).
\end{gather*}
Therefore $S A_{\ell} = A_r S$. To obtain the equations~(\ref{eq:Scom2}), apply the adjoint map to each equation in~(\ref{eq:Scom1}).
\end{proof}

\begin{Lemma}\label{lem:TAB} We have
\begin{alignat}{5}
\label{eq:TAB1}
& T A_\ell = B_\ell T, \qquad &&
 T A_r = B_r T, \qquad &&
 T B_\ell = A_\ell T, \qquad &&
 T B_r = A_r T,&\\
\label{eq:TAB2}
& T A^*_\ell = B^*_\ell T, \qquad &&
 T A^*_r = B^*_r T, \qquad &&
 T B^*_\ell = A^*_\ell T, \qquad &&
 T B^*_r = A^*_r T.
\end{alignat}
\end{Lemma}
\begin{proof} The equations~(\ref{eq:TAB1}) follow from Lemma~\ref{lem:thetaKTS}. We illustrate with a~detailed proof of $T A_\ell = B_\ell T$. For $v \in \mathbb V$,
\begin{gather*}
T A_\ell (v) =
T(A \star v) =
T(A) \star T(v) =
B \star T(v) =
B_\ell T (v).
\end{gather*}
Therefore
$T A_\ell = B_\ell T$. To obtain the equations~(\ref{eq:TAB2}), apply the adjoint map to each equation in~(\ref{eq:TAB1}).
\end{proof}

\begin{Lemma}\label{lem:thetaAB}We have
\begin{alignat}{5}
&\theta A_L = A_\ell \theta,
\qquad &&
\theta A_R = A_r \theta,
\qquad &&
\theta B_L = B_\ell \theta,
\qquad &&
\theta B_R = B_r \theta,&
\label{eq:thetaAB1}
\\
& \theta A^*_\ell = A^*_L \theta,
\qquad &&
\theta A^*_r = A^*_R \theta,
\qquad &&
\theta B^*_\ell = B^*_L \theta,
\qquad &&
\theta B^*_r = B^*_R \theta.&\label{eq:thetaAB2}
\end{alignat}
\end{Lemma}
\begin{proof} The equations~(\ref{eq:thetaAB1}) follow from the definition of $\theta$ below~(\ref{eq:qsc2}). We illustrate with a~detailed proof of $\theta A_L = A_\ell \theta$. For $v \in \mathbb V$,
\begin{gather*}
\theta A_L (v) =
\theta(Av) =
\theta(A) \star \theta(v) =
A \star \theta(v)=
A_\ell \theta (v).
\end{gather*}
Therefore $\theta A_L = A_\ell \theta$. To obtain the equations~(\ref{eq:thetaAB2}), apply the adjoint map to each equation in~(\ref{eq:thetaAB1}).
\end{proof}

\section{Derivations}\label{section8}

In this section we interpret the maps~(\ref{eq:ABLR}),
(\ref{eq:AlBl})
using the derivation concept from
 Section~\ref{section2}. We acknowledge that most if not all of the results in
 this section are well known to the experts; see for example
 \cite[Sections~3 and~4]{green}.

\begin{Lemma}\label{lem:derAs}
For $u,v \in \mathbb V$,
\begin{gather*}
A^*_\ell (uv) =
K(u) A^*_\ell (v)
+A^*_\ell (u)v,
\\
B^*_\ell (uv) =
K^{-1}(u) B^*_\ell (v)+
B^*_\ell (u)v,
\\
A^*_r (uv) =
u A^*_r (v)+
A^*_r (u) K(v),
\\
B^*_r (uv) =
u B^*_r (v)+
B^*_r (u) K^{-1}(v).
\end{gather*}
\end{Lemma}
\begin{proof}Without loss of generality, we may assume that $u$ and $v$ are words in
$\mathbb V$. In this case the result is readily checked using
Lemma~\ref{lem:dualaction}.
\end{proof}

\begin{Corollary}\label{cor:derFree}
For the free algebra $\mathbb V$,
\begin{enumerate}\itemsep=0pt
\item[\rm (i)]
$A^*_\ell$ and
$B^*_r K$ are
 $(K,I)$-derivations;
\item[\rm (ii)]
$B^*_\ell$ and
$A^*_r K^{-1}$ are
$\big(K^{-1},I\big)$-derivations;
\item[\rm (iii)]
$A^*_r$ and
$B^*_\ell K$ are
$(I,K)$-derivations;
\item[\rm (iv)]
$B^*_r$ and
$A^*_\ell K^{-1}$ are
$\big(I,K^{-1}\big)$-derivations.
\end{enumerate}
\end{Corollary}
\begin{proof} By Lemma~\ref{lem:derAs} and the comments about derivations in Section~\ref{section2}.
\end{proof}

\begin{Lemma}\label{lem:circAV}For $u,v \in \mathbb V$,
\begin{gather*}
A^*_L (u \star v) =K(u) \star A^*_L (v)+A^*_L (u) \star v,\\
B^*_L (u\star v) =K^{-1}(u) \star B^*_L (v)+B^*_L (u) \star v,\\
A^*_R (u\star v) =u \star A^*_R (v)+A^*_R (u) \star K(v),\\
B^*_R (u \star v) =u \star B^*_R (v)+B^*_R (u) \star K^{-1}(v).
\end{gather*}
\end{Lemma}
\begin{proof} Without loss, we may assume that $u$ and $v$ are words in~$\mathbb V$. In this case the result is readily checked using~(\ref{eq:uvcirc}), (\ref{eq:uvcirc2}) and Lemma~\ref{lem:Astar}.
\end{proof}

\begin{Corollary}\label{cor:derShuffle}For the $q$-shuffle algebra $\mathbb V$,
\begin{enumerate}\itemsep=0pt
\item[\rm (i)]
$A^*_L$ and
$B^*_R K$ are
$(K,I)$-derivations;
\item[\rm (ii)]
$B^*_L$ and
$A^*_R K^{-1}$ are
 $\big(K^{-1},I\big)$-derivations;
\item[\rm (iii)]
$A^*_R$ and
$B^*_L K$ are
$(I,K)$-derivations;
\item[\rm (iv)]$B^*_R$ and $A^*_L K^{-1}$ are $\big(I,K^{-1}\big)$-derivations.
\end{enumerate}
\end{Corollary}
\begin{proof} By Lemma~\ref{lem:circAV} and the comments about derivations in Section~\ref{section2}.
\end{proof}

\section{Some relations}\label{section9}

Consider the maps $K^{\pm 1}$, (\ref{eq:leftmult}), (\ref{eq:ABLR}), (\ref{eq:alblarbr}), (\ref{eq:AlBl}). In this section we describe how~$K^{\pm 1}$, (\ref{eq:ABLR}),~(\ref{eq:alblarbr}) are related, and how $K^{\pm 1}$, (\ref{eq:leftmult}), (\ref{eq:AlBl}) are related. We also discuss the subspaces~$J$ and~$U$.
We acknowledge that most if not all of the relations in this section are well known to the experts, see for example
\cite[Sections~3.3 and 3.4]{kashiwara}.
\begin{Proposition}\label{thm:v1} The maps
\begin{gather}\label{eq:maplist1}
K, \quad K^{-1},\quad A^*_L,
\quad
B^*_L,
\quad
A^*_R,
\quad
B^*_R,
\quad
A_\ell,
\quad
B_\ell,
\quad
A_r,
\quad
B_r
\end{gather}
satisfy the following relations: $K K^{-1} = K^{-1}K = I$ and
\begin{gather*}
K A^*_L = q^{-2} A^*_L K,\qquad K B^*_L = q^2 B^*_L K,\\
K A^*_R = q^{-2} A^*_R K,\qquad K B^*_R = q^{2} B^*_R K,\\
K A_\ell = q^{2} A_\ell K,\qquad K B_\ell = q^{-2} B_\ell K,\\
K A_r = q^{2} A_r K,\qquad K B_r = q^{-2} B_r K,\\
A^*_L A^*_R = A^*_R A^*_L, \qquad B^*_L B^*_R = B^*_R B^*_L,\\
A^*_L B^*_R = B^*_R A^*_L, \qquad B^*_L A^*_R = A^*_R B^*_L,\\
A_\ell A_r = A_r A_\ell, \qquad B_\ell B_r = B_r B_\ell,\\
A_\ell B_r = B_r A_\ell, \qquad B_\ell A_r = A_r B_\ell,\\
A^*_L B_r = B_r A^*_L, \qquad B^*_L A_r = A_r B^*_L,\\
 A^*_R B_\ell = B_\ell A^*_R, \qquad B^*_R A_\ell = A_\ell B^*_R,\\
A^*_L B_\ell = q^{-2} B_\ell A^*_L, \qquad B^*_L A_\ell = q^{-2} A_\ell B^*_L,\\
A^*_R B_r = q^{-2} B_r A^*_R, \qquad B^*_R A_r = q^{-2} A_r B^*_R,\\
 A^*_L A_\ell - q^2 A_{\ell} A^*_L = I, \qquad A^*_R A_r - q^2 A_r A^*_R = I,\\
 B^*_L B_\ell - q^2 B_{\ell} B^*_L = I,\qquad B^*_R B_r - q^2 B_r B^*_R = I,\\
A^*_L A_r - A_r A^*_L = K,\qquad B^*_L B_r - B_r B^*_L = K^{-1},\\
A^*_R A_\ell - A_\ell A^*_R = K,\qquad B^*_R B_\ell - B_\ell B^*_R = K^{-1},\\
A_\ell^3 B_\ell- \lbrack 3 \rbrack_q A_\ell^2 B_\ell A_\ell+
\lbrack 3 \rbrack_q A_\ell B_\ell A^2_\ell-B_\ell A^3_\ell = 0,\\
B_\ell^3 A_\ell-\lbrack 3 \rbrack_q B_\ell^2 A_\ell B_\ell+
\lbrack 3 \rbrack_q B_\ell A_\ell B^2_\ell-A_\ell B^3_\ell = 0,\\
A_r^3 B_r-\lbrack 3 \rbrack_q A_r^2 B_r A_r+\lbrack 3 \rbrack_q A_r B_r A^2_r-B_r A^3_r = 0,\\
B_r^3 A_r-\lbrack 3 \rbrack_q B_r^2 A_r B_r+\lbrack 3 \rbrack_q B_r A_r B^2_r-A_r B^3_r = 0.
\end{gather*}
\end{Proposition}
\begin{proof}The first 16 relations above are checked by applying each side to a word in $\mathbb V$. The next 16 relations are obtained by setting $u=A$ or $u=B$ or $v=A$ or $v=B$ in Lemma~\ref{lem:circAV}. The last four relations follow from (\ref{eq:qsc1}), (\ref{eq:qsc2}) and the fact that the $q$-shuffle product
is associative.
\end{proof}

\begin{Proposition}\label{thm:v2}The maps
\begin{gather}\label{eq:maplist2}
K, \quad K^{-1},
\quad
A_L,
\quad
B_L,
\quad
A_R,
\quad
B_R,
\quad
A^*_\ell,
\quad
B^*_\ell,
\quad
A^*_r,
\quad
B^*_r
\end{gather}
satisfy the following relations: $K K^{-1} = K^{-1}K = I$ and
\begin{gather*}
K A_L = q^{2} A_L K,\qquad K B_L = q^{-2} B_L K,\\
K A_R = q^{2} A_R K,\qquad K B_R = q^{-2} B_R K,\\
K A^*_\ell = q^{-2} A^*_\ell K,\qquad K B^*_\ell = q^{2} B^*_\ell K,\\
K A^*_r = q^{-2} A^*_r K,\qquad K B^*_r = q^{2} B^*_r K,\\
A_L A_R = A_R A_L, \qquad B_L B_R = B_R B_L,\\
A_L B_R = B_R A_L, \qquad B_L A_R = A_R B_L,\\
A^*_\ell A^*_r = A^*_r A^*_\ell, \qquad B^*_\ell B^*_r = B^*_r B^*_\ell,\\
A^*_\ell B^*_r = B^*_r A^*_\ell, \qquad B^*_\ell A^*_r = A^*_r B^*_\ell,\\
A_L B^*_r = B^*_r A_L, \qquad B_L A^*_r = A^*_r B_L,\\
 A_R B^*_\ell = B^*_\ell A_R, \qquad B_R A^*_\ell = A^*_\ell B_R,\\
A_L B^*_\ell = q^{2} B^*_\ell A_L, \qquad B_L A^*_\ell = q^{2} A^*_\ell B_L,\\
A_R B^*_r = q^{2} B^*_r A_R, \qquad B_R A^*_r = q^{2} A^*_r B_R,\\
 A^*_\ell A_L - q^{2} A_L A^*_\ell = I,\qquad A^*_r A_R - q^2 A_R A^*_r = I,\\
 B^*_\ell B_L - q^2 B_L B^*_\ell = I,\qquad B^*_r B_R - q^2 B_R B^*_r = I,\\
A^*_r A_L - A_L A^*_r = K,\qquad B^*_r B_L - B_L B^*_r = K^{-1},\\
A^*_\ell A_R - A_R A^*_\ell = K,\qquad B^*_\ell B_R - B_R B^*_\ell = K^{-1},\\
(A^*_\ell)^3 B^*_\ell-\lbrack 3 \rbrack_q (A^*_\ell)^2 B^*_\ell A^*_\ell+
\lbrack 3 \rbrack_q A^*_\ell B^*_\ell (A^*_\ell)^2-B^*_\ell (A^*_\ell)^3 = 0,\\
(B^*_\ell)^3 A^*_\ell-\lbrack 3 \rbrack_q (B^*_\ell)^2 A^*_\ell B^*_\ell+
\lbrack 3 \rbrack_q B^*_\ell A^*_\ell (B^*_\ell)^2-A^*_\ell (B^*_\ell)^3 = 0,\\
(A^*_r)^3 B^*_r-\lbrack 3 \rbrack_q (A^*_r)^2 B^*_r A^*_r+\lbrack 3 \rbrack_q A^*_r B^*_r (A^*_r)^2-
B^*_r (A^*_r)^3 = 0,\\
(B^*_r)^3 A^*_r-\lbrack 3 \rbrack_q (B^*_r)^2 A^*_r B^*_r+
\lbrack 3 \rbrack_q B^*_r A^*_r (B^*_r)^2-A^*_r (B^*_r)^3 = 0.
\end{gather*}
\end{Proposition}
\begin{proof} Apply the adjoint map to each relation in Proposition~\ref{thm:v1}.
\end{proof}

\begin{Proposition}\label{prop:Uinv}The subspace $U$ is invariant under each of the maps~\eqref{eq:maplist1}. On $U$,
\begin{gather}\label{eq:one}
(A^*_L)^3 B^*_L
-\lbrack 3 \rbrack_q (A^*_L)^2 B^*_L A^*_L
+\lbrack 3 \rbrack_q A^*_L B^*_L (A^*_L)^2
-B^*_L (A^*_L)^3 = 0,
\\
\label{eq:two}
(B^*_L)^3 A^*_L -\lbrack 3 \rbrack_q (B^*_L)^2 A^*_L B^*_L
+\lbrack 3 \rbrack_q B^*_L A^*_L (B^*_L)^2
-A^*_L (B^*_L)^3 = 0,
\\
\label{eq:three}(A^*_R)^3 B^*_R
-\lbrack 3 \rbrack_q (A^*_R)^2 B^*_R A^*_R
+\lbrack 3 \rbrack_q A^*_R B^*_R (A^*_R)^2
-B^*_R (A^*_R)^3 = 0,\\
\label{eq:four}
(B^*_R)^3 A^*_R
-\lbrack 3 \rbrack_q (B^*_R)^2 A^*_R B^*_R
+\lbrack 3 \rbrack_q B^*_R A^*_R (B^*_R)^2
-A^*_R (B^*_R)^3 = 0.
\end{gather}
\end{Proposition}
\begin{proof}The subspace $U$ is invariant under $K^{\pm 1}$ by Lemma~\ref{lem:KSTU}. The subspace $U$ is invariant under the last eight maps in~(\ref{eq:maplist1}), in view of Lemma~\ref{lem:thetaAB}.
We illustrate with a detailed proof of $A^*_L (U) \subseteq U$. We have
\begin{gather*}
A^*_L (U) = A^*_L \theta (\mathbb V) = \theta A^*_\ell (\mathbb V)
\subseteq \theta (\mathbb V) = U.
\end{gather*}
Next we show that the relation (\ref{eq:one}) holds on $U$.
Let $X$ denote the map on the left in (\ref{eq:one}), and note
that $X^*(v)=-J^+v$ for all $v \in \mathbb V$. We show that
$XU=0$. To do this, it suffices to show that $(XU,\mathbb V)=0$. We have
$(XU,\mathbb V) = (U,X^*\mathbb V)=(U, J^+\mathbb V)$
and $J^+\mathbb V \subseteq J = U^\perp$. Therefore
$(XU,\mathbb V)=0$.
We have shown that the relation (\ref{eq:one}) holds on $U$.
One similarly shows that the relations (\ref{eq:two})--(\ref{eq:four}) hold on $U$.
\end{proof}

\begin{Proposition}\label{prop:Jinv} The subspace $J$ is invariant under each of the maps~\eqref{eq:maplist2}. On the quotient~$\mathbb V/J$,
\begin{gather}\label{eq:One}
A_L^3 B_L-\lbrack 3 \rbrack_q A_L^2 B_L A_L+\lbrack 3 \rbrack_q A_L B_L A_L^2-B_L A_L^3 = 0,\\
\label{eq:Two}
B_L^3 A_L-\lbrack 3 \rbrack_q B_L^2 A_L B_L+\lbrack 3 \rbrack_q B_L A_L B_L^2-A_L B_L^3 = 0,\\
\label{eq:Three}
A_R^3 B_R-\lbrack 3 \rbrack_q A_R^2 B_R A_R+\lbrack 3 \rbrack_q A_R B_R A_R^2-B_R A_R^3 = 0,\\
\label{eq:Four}
B_R^3 A_R-\lbrack 3 \rbrack_q B_R^2 A_R B_R+\lbrack 3 \rbrack_q B_R A_R B_R^2-A_R B_R^3 = 0.
\end{gather}
\end{Proposition}
\begin{proof}The subspace $J$ is invariant under $K^{\pm 1}$ by Lemma~\ref{lem:Jinv}. The subspace $J$ is invariant under $A_L$, $B_L$, $A_R$, $B_R$ since $J$ is a 2-sided ideal of the free algebra $\mathbb V$. The subspace $J$ is invariant under $A^*_\ell$, $B^*_\ell$, $A^*_r$, $B^*_r$ by (\ref{eq:thetaAB2}) and $J={\rm ker}(\theta)$. We verify that the relation~(\ref{eq:One}) holds on $\mathbb V/J$. Let $Y$ denote the map on the left in~(\ref{eq:One}). To show that $Y$ is zero on $\mathbb V/J$, it suffices to show that $Y\mathbb V\subseteq J$. This is the case, since $Y\mathbb V = J^+\mathbb V\subseteq J$. We have verified that~(\ref{eq:One}) holds on $\mathbb V/J$. One similarly verifies that the relations (\ref{eq:Two})--(\ref{eq:Four}) hold on $\mathbb V/J$.
\end{proof}

\section[The algebra $\square^\vee_q$]{The algebra $\boldsymbol{\square^\vee_q}$}\label{section10}

In this section we introduce an algebra $\square^\vee_q$ and describe how it is related to the free algebra~$\mathbb V$. We also discuss the $q$-Serre relations. In the next section we will obtain sixteen $\square^\vee_q$-module structures on~$\mathbb V$. Recall the cyclic group $\mathbb Z_4 = \mathbb Z/4 \mathbb Z$ of order~4.

\begin{Definition}\label{def:squarecheck} Define the algebra $\square^{\vee}_q$ by generators $\lbrace x_i\rbrace_{i \in \mathbb Z_4}$ and relations
\begin{gather}
\frac{q x_i x_{i+1} - q^{-1} x_{i+1} x_i}{q-q^{-1}} = 1, \qquad i \in \mathbb Z_4.\label{eq:checkrels}
\end{gather}
\end{Definition}
\begin{Note}\label{note:square}
The algebra $\widetilde \square_q$ from \cite[Definition~6.1]{pospart} is related to~$\square^{\vee}_q$
in the following way. There exists a surjective algebra homomorphism $\widetilde \square_q \to \square^\vee_q$ that sends $x_i \mapsto x_i $ and $c^{\pm 1}_i \mapsto 1$ for $i \in \mathbb Z_4$.
\end{Note}

The algebra $\square^\vee_q$ is related to the free algebra $\mathbb V$ in the following way. Let $\big(\square^\vee_q\big)^{\rm even}$ (resp.\ $\big(\square^\vee_q\big)^{\rm odd}$) denote the subalgebra of~$\square^\vee_q$ generated by $x_0$, $x_2$ (resp.~$x_1$, $x_3$). Adapting the proof of
\cite[Proposition~6.17]{pospart}, we see that
 \begin{enumerate}\itemsep=0pt
\item[\rm (i)]
 there exists an algebra isomorphism
 $\mathbb V \to \big(\square^\vee_q\big)^{\rm even}$ that sends
 $A\mapsto x_0$ and
 $B\mapsto x_2$;
 \item[\rm (ii)]
 there exists an algebra isomorphism
 $\mathbb V \to \big(\square^\vee_q\big)^{\rm odd}$ that sends
 $A\mapsto x_1$ and
 $B\mapsto x_3$;
 \item[\rm (iii)]
 the multiplication map
 $\big(\square^\vee_q\big)^{\rm even}
 \otimes
 \big(\square^\vee_q\big)^{\rm odd}
 \to \square^\vee_q$,
 $u \otimes v \mapsto uv$ is an isomorphism of vector spaces.
\end{enumerate}

We need a fact about the $q$-Serre relations. We will take a moment to establish this fact, and then return
to the main topic.

\begin{Lemma}\label{lem:fact}
Pick scalars $r,s\in \big\lbrace q^2,q^{-2}\big\rbrace$. Suppose we are given elements $a$, $b$, $x$, $y$, $k$, $k^{-1}$ in any algebra such that $kk^{-1}= k^{-1}k=1$ and
\begin{alignat}{5}
& ax=xa, \qquad&&
ay=ya, \qquad&&
bx=xb, \qquad&&
by=yb,& \label{eq:acom}\\
&ka = rak, \qquad &&
kb = r^{-1} bk, \qquad&&
kx = s xk, \qquad&&
ky = s^{-1} yk.&
\label{eq:kcom}
\end{alignat}
Then
\begin{gather*}
\big(a-k^{-1}x\big)^3 (b-k y)- \lbrack 3 \rbrack_q \big(a-k^{-1}x\big)^2 (b-ky)\big(a-k^{-1}x\big)
\\
\qquad \quad{} + \lbrack 3 \rbrack_q \big(a-k^{-1}x\big) (b-ky)\big(a-k^{-1}x\big)^2 -
(b-ky)\big(a-k^{-1}x\big)^3 \\
\qquad{} = a^3 b - \lbrack 3 \rbrack_q a^2 b a + \lbrack 3 \rbrack_q a b a^2 -ba^3
 + \big(x^3 y -\lbrack 3 \rbrack_q x^2 yx+\lbrack 3 \rbrack_q xy x^2-yx^3\big) k^{-2} s^{-4}
\end{gather*}
and
\begin{gather*}
(b-ky)^3\big(a-k^{-1} x\big)- \lbrack 3 \rbrack_q (b-ky)^2 \big(a-k^{-1}x\big)(b-ky)\\
\qquad \quad {}+ \lbrack 3 \rbrack_q (b-ky) \big(a-k^{-1}x\big)(b-ky)^2 - \big(a-k^{-1}x\big)(b-ky)^3\\
\qquad{} = b^3 a -\lbrack 3 \rbrack_q b^2 a b+\lbrack 3 \rbrack_q b a b^2-ab^3
 + \big(y^3 x -\lbrack 3 \rbrack_q y^2 xy+\lbrack 3 \rbrack_q y x y^2-xy^3\big) k^2 s^{-4}.
\end{gather*}
\end{Lemma}
\begin{proof} To verify each equation, expand the left-hand side and evaluate the result using
(\ref{eq:acom}), (\ref{eq:kcom}).
\end{proof}

\begin{Corollary}\label{cor:aside}With the notation and assumptions of Lemma~{\rm \ref{lem:fact}}, if $a$, $b$ and $x$, $y$ satisfy the $q$-Serre relations, then so do any of the following pairs: {\rm (i)} $a-k^{-1}x$, $b-ky$;
{\rm (ii)} $a-xk^{-1}$, $b-yk$;
{\rm (iii)} $a-kx$, $b-k^{-1}y$;
{\rm (iv)} $a-xk$, $b-yk^{-1}$.
\end{Corollary}
\begin{proof} (i) By Lemma~\ref{lem:fact}.

(ii) Apply (i) above with $a$, $b$ replaced by $s^{-1}a$, $s^{-1}b$.

(iii) Apply (i) above with $k$, $r$, $s$ replaced by $k^{-1}$, $r^{-1}$, $s^{-1}$.

 (iv) Apply (ii) above with $k$, $r$, $s$ replaced by $k^{-1}$, $r^{-1}$, $s^{-1}$.
\end{proof}

\section[Sixteen $\square^\vee_q$-module structures on $\mathbb V$]{Sixteen $\boldsymbol{\square^\vee_q}$-module structures on $\boldsymbol{\mathbb V}$}\label{section11}

In Definition~\ref{def:squarecheck} we defined the algebra $\square^\vee_q$. In this section we describe
sixteen $\square^\vee_q$-module structures on~$\mathbb V$. The first eight involve the maps from~(\ref{eq:maplist2}), and the rest involve the maps from~(\ref{eq:maplist1}). For notational convenience
define $Q=1-q^2$.

\begin{Proposition}\label{prop:fourmod2} For each row in the tables below, the vector space $\mathbb V$ becomes a~$\square^\vee_q$-module on which the generators $\lbrace x_i\rbrace_{i \in \mathbb Z_4}$ act as indicated.

\centerline{
\begin{tabular}[t]{c|cccc}
{\rm module label}& $x_0$ & $x_1$ & $x_2$ & $x_3$
 \\ \hline
{\rm I}
& $A_L$
& $Q(A^*_\ell-B^*_r K)$
& $B_L$
& $Q\big(B^*_\ell-A^*_rK^{-1}\big)$\tsep{2pt}
\\
{\rm IS}
& $A_R$
& $Q(A^*_r-B^*_\ell K)$
& $B_R$
& $Q\big(B^*_r-A^*_\ell K^{-1}\big)$
\\
{\rm IT}
& $B_L$
& $Q\big(B^*_\ell-A^*_rK^{-1}\big)$
& $A_L$
& $Q(A^*_\ell-B^*_r K)$
\\
{\rm IST}
& $B_R$
& $Q\big(B^*_r-A^*_\ell K^{-1}\big)$
& $A_R$
& $Q(A^*_r-B^*_\ell K)$
 \end{tabular}
}
\medskip

\centerline{
\begin{tabular}[t]{c|cccc}
{\rm module label}& $x_0$ & $x_1$ & $x_2$ & $x_3$
 \\ \hline
{\rm II}
 & $Q(A_L- K B_R)$
& $A^*_\ell$
 & $Q\big(B_L-K^{-1}A_R\big)$
 & $B^*_\ell$\tsep{2pt}
\\
{\rm IIS}
& $Q(A_R-KB_L)$
& $A^*_r$
& $Q\big(B_R-K^{-1}A_L\big)$
& $B^*_r$
\\
{\rm IIT}
 & $Q\big(B_L-K^{-1}A_R\big)$
 & $B^*_\ell$
 & $Q(A_L- K B_R)$
& $A^*_\ell$
\\
{\rm IIST}
& $Q\big(B_R-K^{-1}A_L\big)$
& $B^*_r$
& $Q(A_R-KB_L)$
& $A^*_r$
\end{tabular}
}

On each $\square^\vee_q$-module in the tables, the actions of $x_1$, $x_3$ satisfy the $q$-Serre relations.
\end{Proposition}
\begin{proof} By the relations in Proposition~\ref{thm:v2} along with Corollary~\ref{cor:aside}.
\end{proof}

\begin{Proposition}\label{prop:fourmod} For each row in the tables below, the vector space $\mathbb V$ becomes a~$\square^\vee_q$-module on which the generators $\lbrace x_i\rbrace_{i \in \mathbb Z_4}$ act as indicated.

\centerline{
\begin{tabular}[t]{c|cccc}
{\rm module label}& $x_0$ & $x_1$ & $x_2$ & $x_3$
 \\ \hline
{\rm III} &
$A_\ell$ &
$Q(A^*_L-B^*_RK)$ &
$B_\ell$
&$Q\big(B^*_L- A^*_R K^{-1}\big)$\tsep{2pt}
\\
{\rm IIIS} &
$A_r$ &
$Q(A^*_R-B^*_L K)$ &
$B_r$
&
$Q\big(B^*_R-A^*_L K^{-1}\big)$
\\
{\rm IIIT} &
$B_\ell$
&$Q\big(B^*_L- A^*_R K^{-1}\big)$
&
$A_\ell$ &
$Q(A^*_L-B^*_RK)$
\\
{\rm IIIST} &
$B_r$
& $Q\big(B^*_R-A^*_L K^{-1}\big)$
&
$A_r$ &
$Q(A^*_R-B^*_L K)$
 \end{tabular}
}
\medskip

\centerline{
\begin{tabular}[t]{c|cccc}
{\rm module label}& $x_0$ & $x_1$ & $x_2$ & $x_3$
 \\ \hline
{\rm IV} &
$Q(A_\ell-K B_r)$
&
$A^*_L$
&
$Q\big(B_\ell-K^{-1} A_r\big)$ &
$B^*_L$\tsep{2pt}
\\
{\rm IVS}
&
$Q(A_r-KB_\ell)$
&
$A^*_R$
&
$Q\big(B_r-K^{-1} A_\ell\big)$
&
$B^*_R$
\\
{\rm IVT} &
$Q\big(B_\ell-K^{-1} A_r\big)$ &
$B^*_L$ &
$Q(A_\ell-K B_r)$
&
$A^*_L$
\\
{\rm IVST}
&
$Q\big(B_r-K^{-1} A_\ell\big)$
&
$B^*_R$
&
$Q(A_r-KB_\ell)$
&
$A^*_R$
 \end{tabular}
}

On each
 $\square^\vee_q$-module in the tables, the actions of $x_0$, $x_2$ satisfy the $q$-Serre relations.
\end{Proposition}
\begin{proof} By the relations in Proposition~\ref{thm:v1}, along with Corollary~\ref{cor:aside}.
\end{proof}

\begin{Note}Going forward, the $\square^\vee_q$-module $\mathbb V$ with label {\rm I} will be denoted~$\mathbb V_{\rm I}$, and so on.
\end{Note}

We now describe the
 sixteen $\square^\vee_q$-modules in more detail.

\begin{Lemma}\label{lem:updown}For each $\square^\vee_q$-module $\mathbb V$ in Propositions~{\rm \ref{prop:fourmod2}}, {\rm \ref{prop:fourmod}} we have
\begin{gather*}
x_0 \mathbb V_n \subseteq \mathbb V_{n+1}, \qquad
x_1 \mathbb V_n \subseteq \mathbb V_{n-1}, \qquad
x_2 \mathbb V_n \subseteq \mathbb V_{n+1}, \qquad
x_3 \mathbb V_n \subseteq \mathbb V_{n-1}
\end{gather*}
for $n \in \mathbb N$.
\end{Lemma}
\begin{proof} The result for $x_0$ and $x_2$ comes from Lemmas~\ref{lem:ABraise},~\ref{lem:ABlRaise2}.
The result for $x_1$ and $x_3$ comes from Lemmas~\ref{lem:ABsLower},~\ref{lem:ABLower2}.
\end{proof}

\begin{Lemma} \label{lem:dder}For each $\square^\vee_q$-module in Proposition~{\rm \ref{prop:fourmod2}}, the elements $x_1$ and $x_3$ act on the free algebra $\mathbb V$ as a derivation of the following sort:

\centerline{
\begin{tabular}[t]{c|cc}
{\rm module label}& $x_1$ & $x_3$
 \\ \hline
{\rm I, II} &
{\rm $(K,I)$-derivation}
&
{\rm $\big(K^{-1},I\big)$-derivation}\tsep{2pt}
\\
{\rm IS, IIS} &
{\rm $(I,K)$-derivation}
&
{\rm $\big(I, K^{-1}\big)$-derivation}
\\
{\rm IT, IIT} &
{\rm $\big(K^{-1},I\big)$-derivation}
&
{\rm $( K,I)$-derivation}
\\
{\rm IST, IIST} &
{\rm $\big(I,K^{-1}\big)$-derivation}
&
{\rm $(I,K)$-derivation}
 \end{tabular}
}
\end{Lemma}
\begin{proof} By Corollary~\ref{cor:derFree}.
\end{proof}

\begin{Lemma} \label{lem:dder2}For each $\square^\vee_q$-module in Proposition~{\rm \ref{prop:fourmod}}, the elements $x_1$ and $x_3$ act on the $q$-shuffle algebra $\mathbb V$ as a~derivation of the following sort:

\centerline{
\begin{tabular}[t]{c|cc}
{\rm module label}& $x_1$ & $x_3$
 \\ \hline
{\rm III, IV} &
{\rm $(K,I)$-derivation}
&
{\rm $\big(K^{-1},I\big)$-derivation}\tsep{2pt}
\\
{\rm IIIS, IVS} &
{\rm $(I,K)$-derivation}
&
{\rm $\big(I, K^{-1}\big)$-derivation}
\\
{\rm IIIT, IVT} &
{\rm $\big(K^{-1},I\big)$-derivation}
&
{\rm $( K,I)$-derivation}
\\
{\rm IIIST, IVST} &
{\rm $\big(I,K^{-1}\big)$-derivation}
&
{\rm $(I,K)$-derivation}
 \end{tabular}
}
\end{Lemma}
\begin{proof} By Corollary \ref{cor:derShuffle}.
\end{proof}

\section[Some homomorphisms between the sixteen $\square^\vee_q$-modules]{Some homomorphisms between the sixteen $\boldsymbol{\square^\vee_q}$-modules}\label{section12}

In the previous section we gave sixteen $\square^\vee_q$-module structures on~$\mathbb V$. In this section
we describe some homomorphisms between them.

\begin{Lemma}\label{lem:SIso}The map $S \in {\rm End}(\mathbb V)$ is an isomorphism of
 $\square^\vee_q$-modules from
\begin{alignat*}{5}
&
\mathbb V_{\rm I} \leftrightarrow \mathbb V_{\rm IS}, \qquad &&
\mathbb V_{\rm IT} \leftrightarrow \mathbb V_{\rm IST}, \qquad &&
\mathbb V_{\rm II} \leftrightarrow \mathbb V_{\rm IIS}, \qquad &&
\mathbb V_{\rm IIT} \leftrightarrow \mathbb V_{\rm IIST}, &
\\
&\mathbb V_{\rm III} \leftrightarrow \mathbb V_{\rm IIIS}, \qquad &&
\mathbb V_{\rm IIIT} \leftrightarrow \mathbb V_{\rm IIIST}, \qquad &&
\mathbb V_{\rm IV} \leftrightarrow \mathbb V_{\rm IVS}, \qquad &&
\mathbb V_{\rm IVT} \leftrightarrow \mathbb V_{\rm IVST}.&
\end{alignat*}
\end{Lemma}
\begin{proof} By Lemmas \ref{lem:scom}, \ref{lem:SAB} and since $S$ is a bijection.
\end{proof}

\begin{Lemma}\label{lem:TIso}The map $T \in {\rm End}(\mathbb V)$ is an isomorphism of $\square^\vee_q$-modules from
\begin{alignat*}{5}
&
\mathbb V_{\rm I} \leftrightarrow \mathbb V_{\rm IT}, \qquad &&
\mathbb V_{\rm IS} \leftrightarrow\mathbb V_{\rm IST}, \qquad &&
\mathbb V_{\rm II} \leftrightarrow \mathbb V_{\rm IIT}, \qquad &&
\mathbb V_{\rm IIS} \leftrightarrow \mathbb V_{\rm IIST},&
\\
&\mathbb V_{\rm III} \leftrightarrow \mathbb V_{\rm IIIT}, \qquad &&
\mathbb V_{\rm IIIS} \leftrightarrow \mathbb V_{\rm IIIST}, \qquad &&
\mathbb V_{\rm IV} \leftrightarrow \mathbb V_{\rm IVT}, \qquad &&
\mathbb V_{\rm IVS} \leftrightarrow \mathbb V_{\rm IVST}.&
\end{alignat*}
\end{Lemma}
\begin{proof} By Lemmas \ref{lem:Tact}, \ref{lem:TAB} and since $T$ is a bijection.
\end{proof}

\begin{Lemma}\label{lem:thetamix} The map $\theta \in {\rm End}(\mathbb V)$ is a homomorphism of $\square^\vee_q$-modules from
\begin{alignat*}{5}
&
\mathbb V_{\rm I} \rightarrow \mathbb V_{\rm III}, \qquad &&
\mathbb V_{\rm IS} \rightarrow \mathbb V_{\rm IIIS}, \qquad &&
\mathbb V_{\rm IT} \rightarrow\mathbb V_{\rm IIIT}, \qquad &&
\mathbb V_{\rm IST} \rightarrow \mathbb V_{\rm IIIST},&
\\
&\mathbb V_{\rm II} \rightarrow \mathbb V_{\rm IV}, \qquad &&
\mathbb V_{\rm IIS} \rightarrow \mathbb V_{\rm IVS}, \qquad &&
\mathbb V_{\rm IIT} \rightarrow \mathbb V_{\rm IVT}, \qquad &&
\mathbb V_{\rm IIST} \rightarrow \mathbb V_{\rm IVST}.&
\end{alignat*}
\end{Lemma}
\begin{proof} By Lemma \ref{lem:thetaAB} and the first equation in Lemma~\ref{lem:thetaKTS2}.
\end{proof}

Our next goal is to display a~map $\varphi \in {\rm End}(\mathbb V)$ such that $\varphi$ is a~$\square^\vee_q$-module isomorphism from $\mathbb V_{\rm I}\to \mathbb V_{\rm II}$ and $\mathbb V_{\rm IT}\to \mathbb V_{\rm IIT}$, and $\varphi^*$ is a $\square^\vee_q$-module isomorphism from $\mathbb V_{\rm III}\to \mathbb V_{\rm IV}$ and $\mathbb V_{\rm IIIT}\to \mathbb V_{\rm IVT}$.

\begin{Definition}\label{def:sigmaIso}Define a map $\varphi \in {\rm End}(\mathbb V)$ as follows. For a word $v=v_1v_2\cdots v_n$ in $\mathbb V$ we have
\begin{gather}\label{eq:sigIso}
\varphi(v) =\widehat v_1\widehat v_2\cdots\widehat v_n (1),
\end{gather}
where
\begin{gather}\label{eq:widehatAB}
\widehat A = Q(A_L-KB_R), \qquad
\widehat B = Q\big(B_L-K^{-1}A_R\big).
\end{gather}
In particular $\varphi (1)=1$.
\end{Definition}

\begin{Lemma}\label{lem:sigmaAct}We have $\varphi \mathbb V_n \subseteq \mathbb V_n$ for $n\in \mathbb N$.
\end{Lemma}
\begin{proof} By Lemma~\ref{lem:ABraise} and Definition~\ref{def:sigmaIso}.
\end{proof}

Shortly we will describe $\varphi$ from an another point of view. In this description we use the following notation. Let $v=v_1v_2\cdots v_n$ denote a word in $\mathbb V$. For a subset $\Omega \subseteq \lbrace 1,2,\ldots, n\rbrace$ we define a word $v_\Omega$ as follows. Write $\Omega = \lbrace i_1,i_2,\ldots, i_k\rbrace $ with $i_1 < i_2 < \cdots < i_k$. Then $v_\Omega = v_{i_1} v_{i_2} \cdots v_{i_k}$. Let $\overline \Omega$ denote the complement of $\Omega$ in $\lbrace 1,2,\ldots, n\rbrace$.
The word $v_{\Omega}$ is obtained from $v_1v_2\cdots v_n$ by deleting $v_j$ for each $j \in \overline \Omega$. Note that $v_\varnothing = 1$.

\begin{Lemma} For a word $v=v_1v_2\cdots v_n$ in $\mathbb V$,
\begin{gather*}
\varphi(v) = Q^n \sum_{\Omega}
v_{\overline \Omega} ST(v_\Omega) (-1)^{|\Omega|} q^{-2|\Omega|}
\Biggl(
\prod_{\stackrel{i,j\in \Omega}{i<j}} q^{-\langle v_i,v_j\rangle}
\Biggr)
\Biggl(
\prod_{\stackrel{i\in \Omega, \;j\in \overline \Omega}
{i<j}} q^{\langle v_i,v_j\rangle}
\Biggr),
\end{gather*}
where the sum is over all subsets $\Omega $ of $\lbrace 1,2,\ldots, n\rbrace$. The maps $S$ and $T$ are from
Definitions~{\rm \ref{def:S}} and~{\rm \ref{def:T}}, respectively.
\end{Lemma}

\begin{proof} Expand the right-hand side of (\ref{eq:sigIso}) using~(\ref{eq:widehatAB}).
\end{proof}

\begin{Lemma}\label{lem:TT} We have $ T \varphi = \varphi T$ and $ T \varphi^* = \varphi^* T$.
\end{Lemma}
\begin{proof}We first show that $ T \varphi = \varphi T$. By Definition~\ref{def:T}, $T(A)=B$ and $T(B)=A$.
By Lemma~\ref{lem:Tact} and~(\ref{eq:widehatAB}), $T \widehat A = \widehat B T$ and $T \widehat B = \widehat A T$. By these comments $T \widehat A = \widehat {T(A)} T$ and $T \widehat B = \widehat {T(B)} T$.
We show that $ T \varphi (v) = \varphi T (v)$ for all $v \in \mathbb V$. Without loss of generality,
we may assume that $v$ is a word in $\mathbb V$. Write $v=v_1v_2\cdots v_n$. By Definition~\ref{def:sigmaIso} and the above comments,
\begin{align*}
T \varphi (v) &= T \varphi (v_1v_2\cdots v_n) =T \widehat v_1 \widehat v_2\cdots \widehat v_n (1)=
 \widehat {T(v_1)} \widehat {T(v_2)}\cdots \widehat {T(v_n)} T (1) \\
& = \widehat {T(v_1)} \widehat {T(v_2)}\cdots \widehat {T(v_n)} (1)=\varphi \bigl(T(v_1) T(v_2)\cdots T(v_n)\bigr)=\varphi T (v_1v_2\cdots v_n) =\varphi T (v).
\end{align*}
We have shown that $ T \varphi = \varphi T$. In this equation, apply the adjoint map to each side and
use $T^*=T$ to obtain $ T \varphi^* = \varphi^* T$.
\end{proof}

\begin{Lemma}\label{lem:braid}On $\mathbb V$,
\begin{alignat}{3}
&\varphi A_L = Q(A_L-K B_R) \varphi, \qquad && \varphi B_L = Q\big(B_L-K^{-1} A_R\big) \varphi,& \label{eq:braid}\\
& \varphi (A^*_\ell-B^*_r K)Q = A^*_\ell \varphi, \qquad&& \varphi \big(B^*_\ell-A^*_r K^{-1}\big)Q = B^*_\ell \varphi.& \label{eq:2braid}
\end{alignat}
\end{Lemma}
\begin{proof} We first obtain~(\ref{eq:braid}). For $x \in \lbrace A, B\rbrace$ we show that $\varphi x_L = \widehat x \varphi$, where $\widehat x$ is from~(\ref{eq:widehatAB}). It suffices to show that $\varphi x_L (v) = \widehat x \varphi (v)$ for all words $v$ in $\mathbb V$. Let the word $v$ be given, and write $v=v_1v_2\cdots v_n$. Using~(\ref{eq:sigIso}),
\begin{gather*}
\varphi x_L (v) =\varphi x_L (v_1v_2\cdots v_n)=\varphi (x v_1v_2\cdots v_n)
= \widehat x \widehat{v_1} \widehat{v_2}\cdots \widehat{v_n}(1)= \widehat x \varphi (v).
\end{gather*}
We have obtained (\ref{eq:braid}). Next we obtain the equation on the left in~(\ref{eq:2braid}).
For that equation let $\Delta$ denote the left-hand side minus the right-hand side. We show that $\Delta=0$.
For notational convenience define
\begin{alignat*}{4}
&x = A_L, \qquad && y= Q(A^*_{\ell}-B^*_r K), \qquad && z= B_L,& \\
&X= Q(A_L-KB_R), \qquad && Y=A^*_\ell, \qquad && Z= Q\big(B_L-K^{-1} A_R\big).
\end{alignat*}
Note that $\Delta = \varphi y - Y \varphi$. Referring to Proposition~\ref{prop:fourmod2}, from the $\mathbb V_{\rm I}$ data
\begin{gather*}
\frac{qxy-q^{-1}yx}{q-q^{-1}}=1, \qquad \frac{qyz-q^{-1}zy}{q-q^{-1}}=1,
\end{gather*}
and from the $\mathbb V_{\rm II}$ data
\begin{gather*}
\frac{qXY-q^{-1}YX}{q-q^{-1}}=1, \qquad \frac{qYZ-q^{-1}ZY}{q-q^{-1}}=1.
\end{gather*}
By (\ref{eq:braid}),
\begin{gather*}
\varphi x = X \varphi, \qquad \varphi z = Z \varphi.
\end{gather*}
We will show that
\begin{gather}
X\Delta = q^{-2} \Delta x, \qquad Z\Delta = q^2 \Delta z.
\label{eq:XZD}
\end{gather}
We have
\begin{gather*}
X \Delta = X \varphi y - X Y \varphi
= \varphi x y - X Y \varphi
= q^{-2}(\varphi y x - Y X \varphi)
= q^{-2}(\varphi y x - Y \varphi x)
= q^{-2}\Delta x.
\end{gather*}
Similarly
\begin{gather*}
Z \Delta = Z \varphi y - Z Y \varphi
= \varphi z y - Z Y \varphi
= q^{2}(\varphi y z - Y Z \varphi)
= q^{2}(\varphi y z - Y \varphi z)
= q^{2}\Delta z.
\end{gather*}
We have shown (\ref{eq:XZD}). We can now easily show that $\Delta = 0$. We define $W=\lbrace v \in \mathbb V\,|\, \Delta v = 0\rbrace$ and show that $W=\mathbb V$. By~(\ref{eq:XZD}), $W$ is invariant under $x=A_L$ and $z=B_L$. Note that $\Delta (1)=0$, since $y(1)=0$, $Y(1)=0$, $\varphi (1)=1$. Therefore $1 \in W$.
By these comments and Lemma~\ref{lem:WAB} we obtain $W=\mathbb V$, so $\Delta=0$. We have obtained
the equation on the left in~(\ref{eq:2braid}). In this equation, multiply each side on the left by $T$ and on the right by~$T^{-1}$. Simplify the result using the first equation in Lemma~\ref{lem:TT}, together with the equations
\begin{gather*}
TKT^{-1}= K^{-1},\qquad T A^*_\ell T^{-1} = B^*_{\ell},\qquad T B^*_r T^{-1} = A^*_r
\end{gather*}
from Lemmas \ref{lem:Tact}, \ref{lem:TAB}. This yields the equation on the right in~(\ref{eq:2braid}).
\end{proof}

\begin{Corollary}\label{cor:hom}The map $\varphi$ from Definition~{\rm \ref{def:sigmaIso}} is a homomorphism of $\square^\vee_q$-modules from $\mathbb V_{\rm I} \to \mathbb V_{\rm II}$ and $\mathbb V_{\rm IT} \to \mathbb V_{\rm IIT}$.
\end{Corollary}
\begin{proof}The $\square^\vee_q$-modules in the corollary statement are described in Proposition~\ref{prop:fourmod2}. Using these descriptions and Lemma~\ref{lem:braid}, we find that for $i \in \mathbb Z_4$ the relation $ \varphi x_i = x_i \varphi$ holds on $\mathbb V_{\rm I}$ and $\mathbb V_{\rm IT}$. The result follows.
\end{proof}

Next we consider the adjoint $\varphi^*$.

\begin{Lemma}\label{lem:adjphi} We have $\varphi^* \mathbb V_n \subseteq \mathbb V_n$ for $n \in \mathbb N$. Moreover $\varphi^*(1)=1$.
\end{Lemma}
\begin{proof} The first assertion follows from Lemma~\ref{lem:sigmaAct} and the fact that the summands in~(\ref{eq:grading}) are mutually orthogonal. To obtain the last assertion, use $\varphi(1)=1$ and the
fact that $1$ is a basis for $\mathbb V_0$.
\end{proof}

\begin{Lemma}\label{lem:braiddual}On $\mathbb V$,
\begin{alignat*}{3}
& A^*_L\varphi^*=\varphi^*(A^*_L- B^*_R K)Q,\qquad &&
B^*_L \varphi^*=\varphi^* \big(B^*_L-A^*_R K^{-1}\big)Q, &\\ % \label{eq:braiddual}\\
&Q(A_\ell-K B_r ) \varphi^*=\varphi^*A_\ell,\qquad && Q\big(B_\ell- K^{-1}A_r\big)
\varphi^*=\varphi^*B_\ell.&%\label{eq:2braiddual}
\end{alignat*}
\end{Lemma}
\begin{proof}For each equation in Lemma~\ref{lem:braid}, apply the adjoint map to each side.
\end{proof}

\begin{Corollary}\label{cor:hom2}The map $\varphi^*$ is a homomorphism of $\square^\vee_q$-modules from
 $\mathbb V_{\rm III} \to \mathbb V_{\rm IV}$ and $\mathbb V_{\rm IIIT} \to \mathbb V_{\rm IVT}$.
\end{Corollary}
\begin{proof}The $\square^\vee_q$-modules in the corollary statement are described in Proposition~\ref{prop:fourmod}. Using these descriptions and Lemma~\ref{lem:braiddual}, we find that for $i \in \mathbb Z_4$ the relation $ \varphi^* x_i = x_i \varphi^*$ holds on~$\mathbb V_{\rm III}$ and~$\mathbb V_{\rm IIIT}$. The result follows.
\end{proof}

\begin{Lemma}\label{lem:thphi}We have $\varphi^* \theta = \theta \varphi$.
\end{Lemma}
\begin{proof}We define $\Delta = \varphi^* \theta-\theta \varphi$ and show that $\Delta=0$. By Lemma~\ref{lem:thetamix} and Corolla\-ries~\ref{cor:hom},~\ref{cor:hom2} we find that each of~$\varphi^* \theta $ and $\theta \varphi $ is a~$\square^\vee_q$-module homomorphism $\mathbb V_{\rm I} \to \mathbb V_{\rm IV}$. So $\Delta$ is a~$\square^\vee_q$-module homomorphism $\mathbb V_{\rm I} \to \mathbb V_{\rm IV}$. By Proposition~\ref{prop:fourmod2}, $x_0$ and $x_2$ act on $\mathbb V_{\rm I}$ as $A_L$ and $B_L$, respectively. By Proposition~\ref{prop:fourmod}, $x_0$ and $x_2$ act on $\mathbb V_{\rm IV}$ as $Q (A_\ell-K B_r)$ and $Q \big(B_\ell-K^{-1} A_r\big)$, respectively. By these comments
\begin{gather}
\Delta A_L = Q (A_\ell-K B_r) \Delta,\qquad \Delta B_L = Q \big(B_\ell-K^{-1} A_r\big) \Delta.\label{eq:AB}
\end{gather}
Let $W$ denote the kernel of $\Delta$ on $\mathbb V$. By~(\ref{eq:AB}), $W$ is invariant under $A_L$ and $B_L$. We have $\Delta(1)=0$ since $\theta(1)=1$, $\varphi(1)=1$, $\varphi^*(1)=1$. Therefore $1 \in W$. By these comments and Lemma~\ref{lem:WAB} we obtain $W=\mathbb V$, so $\Delta=0$.
\end{proof}

Next we show that $\varphi$ and $\varphi^*$ are bijections.

\begin{Definition}\label{def:varphi} Define $\psi \in {\rm End}(\mathbb V)$ by
\begin{gather}\label{eq:varphi}
\psi = K B_R A^*_L + K^{-1} A_R B^*_L.
\end{gather}
\end{Definition}

\begin{Lemma}We have $\psi \mathbb V_0=0$, and $\psi \mathbb V_n \subseteq \mathbb V_n$ for $n \geq 1$.
\end{Lemma}
\begin{proof} By Lemmas \ref{lem:ABraise}, \ref{lem:ABsLower}.
\end{proof}

\begin{Lemma}\label{lem:varphiAct} The map $\psi $ acts on the standard basis for $\mathbb V$ as follows. We have $\psi(1)=0$. For a nontrivial word $v=v_1 v_2 \cdots v_n$ in $\mathbb V$,
\begin{gather*}
\psi(v) = v_2 v_3 \cdots v_n T(v_1)q^{\langle v_1, v_2\rangle+\langle v_1, v_3\rangle+\cdots + \langle v_1, v_n\rangle-2}.
\end{gather*}
\end{Lemma}
\begin{proof} Use Definition \ref{def:Kdef} and the comments below it, along with Lemma~\ref{lem:Astar} and Definition~\ref{def:varphi}.
\end{proof}

\begin{Lemma}\label{lem:spam} The following {\rm (i)--(iv)} hold for $n\geq 1$.
\begin{enumerate}\itemsep=0pt
\item[\rm (i)]
The equation
$\psi^n = q^{-2n} T$ holds
on $\mathbb V_n$.
\item[\rm (ii)] The equation
$\psi^{2n} = q^{-4n} I$ holds
on $\mathbb V_n$.
\item[\rm (iii)]
The map $I-\psi$ is invertible on $\mathbb V_n$.
\item[\rm (iv)] On $\mathbb V_n$,
\begin{gather*}
I-\psi= (A_L-K B_R)A^*_L +\big(B_L-K^{-1} A_R\big)B^*_L.
\end{gather*}
\end{enumerate}
\end{Lemma}
\begin{proof} (i) Pick a word $v=v_1v_2 \cdots v_n$ in $\mathbb V$. Repeatedly applying Lemma~\ref{lem:varphiAct} we obtain
\begin{gather*}
\psi^n(v) = T(v_1) T(v_2) \cdots T(v_n) q^{-2n} = T(v) q^{-2n}.
\end{gather*}
The result follows.

(ii) By (i) above and since $T^2=I$.

(iii) Pick a vector $v\in \mathbb V_n$ such that $(I-\psi)v=0$. We show that $v=0$. By construction $\psi v = v$. By this and~(ii) we obtain $v=\psi^{2n}v = q^{-4n} v$. So $\big(1-q^{-4n}\big)v=0$. In this equation the coefficient of $v$ is nonzero, so $v=0$.

(iv) By~(\ref{eq:varphi}) and since $I = A_L A^*_L + B_L B^*_L$ on $\mathbb V_n$.
\end{proof}

\begin{Lemma}\label{lem:span} For $n\geq 1$ we have
\begin{gather*}
\mathbb V_n = (A_L - K B_R)\mathbb V_{n-1} + \big(B_L-K^{-1} A_R\big)\mathbb V_{n-1}.%\label{eq:span}
\end{gather*}
\end{Lemma}
\begin{proof}The inclusion $\supseteq $ follows from Lemma~\ref{lem:ABraise}. Concerning the inclusion $\subseteq$, use parts (iii), (iv) of Lemma~\ref{lem:spam} along with Lemma~\ref{lem:ABsLower} to obtain
\begin{align*}
\mathbb V_n &= (I-\psi)\mathbb V_n =\bigl((A_L-K B_R)A^*_L+ \big(B_L-K^{-1} A_R\big)B^*_L\bigr) \mathbb V_n\\
&\subseteq(A_L-K B_R)A^*_L \mathbb V_n + \big(B_L-K^{-1} A_R\big)B^*_L\mathbb V_n\\
&\subseteq (A_L-K B_R) \mathbb V_{n-1}+ \big(B_L-K^{-1} A_R\big) \mathbb V_{n-1}.\tag*{\qed}
\end{align*}\renewcommand{\qed}{}
\end{proof}

\begin{Lemma}\label{lem:qroot}For $n \in \mathbb N$ we have
\begin{enumerate}\itemsep=0pt
\item[\rm (i)]
 $\varphi \mathbb V_n = \mathbb V_n$;
\item[\rm (ii)]
 $\varphi^* \mathbb V_n = \mathbb V_n$.
 \end{enumerate}
\end{Lemma}
\begin{proof}(i) By induction on $n$. For $n=0$ the result holds, because $1$ is a basis for $\mathbb V_0$ and $\varphi(1)=1$. Next assume that $n\geq 1$. The sum $\mathbb V_n= A \mathbb V_{n-1} + B \mathbb V_{n-1}$ is direct. In this equation apply $\varphi$ to each side. By Lemma~\ref{lem:braid} and induction,
\begin{gather*}
\varphi (A \mathbb V_{n-1})=\varphi A_L \mathbb V_{n-1}= (A_L - K B_R)\varphi \mathbb V_{n-1}
= (A_L - K B_R)\mathbb V_{n-1}, \\
\varphi (B \mathbb V_{n-1}) =\varphi B_L \mathbb V_{n-1} = \big(B_L - K^{-1} A_R\big)\varphi \mathbb V_{n-1}
= \big(B_L - K^{-1} A_R\big)\mathbb V_{n-1}.
\end{gather*}
By these comments and Lemma~\ref{lem:span},
\begin{gather*}
\varphi \mathbb V_n =\varphi (A \mathbb V_{n-1}) + \varphi (B \mathbb V_{n-1})
= (A_L - K B_R)\mathbb V_{n-1} +\big(B_L - K^{-1} A_R\big)\mathbb V_{n-1} = \mathbb V_n.
\end{gather*}

(ii) By Lemma~\ref{lem:adjphi} and (i) above, along with the fact that the dimension of $\mathbb V_n$ is finite.
\end{proof}

\begin{Lemma}\label{lem:bij}Each of $\varphi$, $\varphi^*$ is a bijection.
\end{Lemma}
\begin{proof}By~(\ref{eq:grading}) and Lemma~\ref{lem:qroot}.
\end{proof}

\begin{Proposition}\label{prop:expand}The map $\varphi$ is an isomorphism of $\square^\vee_q$-modules from
$\mathbb V_{\rm I} \to \mathbb V_{\rm II}$ and $\mathbb V_{\rm IT} \to\mathbb V_{\rm IIT}$. Moreover $\varphi^*$ is an isomorphism of $\square^\vee_q$-modules from $\mathbb V_{\rm III} \to \mathbb V_{\rm IV}$ and $\mathbb V_{\rm IIIT} \to \mathbb V_{\rm IVT}$.
\end{Proposition}
\begin{proof} By Corollaries~\ref{cor:hom},~\ref{cor:hom2} and Lemma~\ref{lem:bij}.
\end{proof}

\begin{Lemma}\label{lem:phiJU}We have $\varphi (J) = J$.
\end{Lemma}
\begin{proof} We invoke Lemmas~\ref{lem:thphi},~\ref{lem:bij} and $J={\rm ker}(\theta)$. For $v \in \mathbb V$,
\begin{gather*}
v \in J
\ \Leftrightarrow\
\theta(v) =0
\ \Leftrightarrow\
\varphi^* \theta(v) =0
\ \Leftrightarrow\
\theta \varphi(v) =0
\ \Leftrightarrow\
\varphi(v) \in J.
\end{gather*}
The result follows.
\end{proof}

\begin{Lemma}\label{lem:twoincl}We have $ \varphi^* (U) = U$.
\end{Lemma}
\begin{proof}By Lemmas~\ref{lem:thphi}, \ref{lem:bij} and $\theta(\mathbb V)=U$,
\begin{gather*}
\varphi^*(U) = \varphi^* \theta (\mathbb V) = \theta \varphi (\mathbb V)= \theta (\mathbb V) = U.\tag*{\qed}
\end{gather*}\renewcommand{\qed}{}
\end{proof}

\section[Sixteen $\square_q$-module structures on $\mathbb V/J$ or $U$]{Sixteen $\boldsymbol{\square_q}$-module structures on $\boldsymbol{\mathbb V/J}$ or $\boldsymbol{U}$}\label{section13}

In Section~\ref{section1} we informally discussed the algebra $\square_q$. In this section we formally bring in~$\square_q$, and review its basic properties. We then display sixteen $\square_q$-module structures on $\mathbb V/J$ or~$U$. In order to motivate things, we mention a result about the $\square^\vee_q$-modules from Propositions~\ref{prop:fourmod2} and~\ref{prop:fourmod}.

\begin{Lemma}\label{lem:alwaysUmod}
The following {\rm (i)}, {\rm (ii)} hold.
\begin{enumerate}\itemsep=0pt
\item[\rm (i)] For each $\square^\vee_q$-module $\mathbb V$ in Proposition~{\rm \ref{prop:fourmod2}},
 the subspace $J$ is a~$\square^\vee_q$-submodule. On the quotient $\square^\vee_q$-module $\mathbb V/J$
the actions of $x_0$, $x_2$ satisfy the $q$-Serre relations.
\item[\rm (ii)] For each $\square^\vee_q$-module $\mathbb V$ from Proposition~{\rm \ref{prop:fourmod}},
 the subspace $U$ is a~$\square^\vee_q$-submodule on which the actions of $x_1$, $x_3$ satisfy the $q$-Serre relations.
\end{enumerate}
\end{Lemma}
\begin{proof} (i) Use Propositions~\ref{prop:Jinv}, \ref{prop:fourmod2} and Corollary~\ref{cor:aside}.

 (ii) Use Propositions~\ref{prop:Uinv}, \ref{prop:fourmod} and Corollary~\ref{cor:aside}.
\end{proof}

\begin{Definition}[{see \cite[Definition~5.1]{pospart}}]\label{def:square} Define the algebra $\square_q$ by generators $\lbrace x_i \rbrace_{i \in \mathbb Z_4}$ and relations
\begin{gather*}
\frac{q x_i x_{i+1}-q^{-1}x_{i+1}x_i}{q-q^{-1}} = 1,\\ %\label{eq:rel1}\\
x^3_i x_{i+2} - \lbrack 3 \rbrack_q x^2_i x_{i+2} x_i +\lbrack 3 \rbrack_q x_i x_{i+2} x^2_i -x_{i+2} x^3_i = 0.%\label{eq:rel2}
\end{gather*}
\end{Definition}

\begin{Lemma}\label{lem:add}There exists an algebra homomorphism $\square^\vee_q \to \square_q$ that sends
$x_i \mapsto x_i$ for $i \in {\mathbb Z}_4$. This homomorphism is surjective.
\end{Lemma}
\begin{proof} Compare Definitions~\ref{def:squarecheck}, \ref{def:square}.
\end{proof}

Recall the algebra $U^+_q=\mathbb V/J$ from Definition~\ref{def:pp}. This algebra is related to~$\square_q$ in the following way. Let $ \square^{\rm even}_q$ (resp.\ $ \square^{\rm odd}_q$) denote the subalgebra
of $ \square_q$ generated by $x_0$, $x_2$ (resp.\ $x_1$, $x_3$). Then by \cite[Proposition~5.5]{pospart},
 \begin{enumerate}\itemsep=0pt
\item[\rm (i)] there exists an algebra isomorphism $U^+_q \to \square^{\rm even}_q$ that sends
 $A\mapsto x_0$ and $B\mapsto x_2$;
\item[\rm (ii)] there exists an algebra isomorphism $U^+_q \to \square^{\rm odd}_q$ that sends $A\mapsto x_1$ and $B\mapsto x_3$;
\item[\rm (iii)] the multiplication map $\square^{\rm even}_q \otimes \square^{\rm odd}_q \to \square_q$,
 $u \otimes v \mapsto uv$ is an isomorphism of vector spaces.
\end{enumerate}

\begin{Theorem}\label{prop:Fourmod2main} For each row in the tables below, the vector space $\mathbb V/J$ becomes a~$\square_q$-module on which the generators $\lbrace x_i\rbrace_{i \in \mathbb Z_4}$ act as indicated.

\centerline{
\begin{tabular}[t]{c|cccc}
{\rm module label}& $x_0$ & $x_1$ & $x_2$ & $x_3$
 \\ \hline
{\rm I}
& $A_L$
& $Q(A^*_\ell-B^*_r K)$
& $B_L$
& $Q\big(B^*_\ell-A^*_rK^{-1}\big)$\tsep{2pt}
\\
{\rm IS}
& $A_R$
& $Q(A^*_r-B^*_\ell K)$
& $B_R$
& $Q\big(B^*_r-A^*_\ell K^{-1}\big)$
\\
{\rm IT}
& $B_L$
& $Q\big(B^*_\ell-A^*_rK^{-1}\big)$
& $A_L$
& $Q(A^*_\ell-B^*_r K)$
\\
{\rm IST}
& $B_R$
& $Q\big(B^*_r-A^*_\ell K^{-1}\big)$
& $A_R$
& $Q(A^*_r-B^*_\ell K)$
 \end{tabular}
}
\medskip

\centerline{
\begin{tabular}[t]{c|cccc}
{\rm module label}& $x_0$ & $x_1$ & $x_2$ & $x_3$
 \\ \hline
{\rm II}
 & $Q(A_L- K B_R)$
& $A^*_\ell$
 & $Q\big(B_L-K^{-1}A_R\big)$
 & $B^*_\ell$\tsep{2pt}
\\
{\rm IIS}
& $Q(A_R-KB_L)$
& $A^*_r$
& $Q\big(B_R-K^{-1}A_L\big)$
& $B^*_r$
\\
{\rm IIT}
 & $Q\big(B_L-K^{-1}A_R\big)$
 & $B^*_\ell$
 & $Q(A_L- K B_R)$
& $A^*_\ell$
\\
{\rm IIST}
& $Q\big(B_R-K^{-1}A_L\big)$
& $B^*_r$
& $Q(A_R-KB_L)$
& $A^*_r$
\end{tabular}
}
\end{Theorem}
\begin{proof}By Proposition~\ref{prop:fourmod2} and Lemma~\ref{lem:alwaysUmod}(i).
\end{proof}

\begin{Theorem}\label{prop:Fourmodmain}For each row in the tables below, the vector space $U$ becomes a
 $\square_q$-module on which the generators $\lbrace x_i\rbrace_{i \in \mathbb Z_4}$ act as indicated.

\centerline{
\begin{tabular}[t]{c|cccc}
{\rm module label}& $x_0$ & $x_1$ & $x_2$ & $x_3$
 \\ \hline
{\rm III} &
$A_\ell$ &
$Q(A^*_L-B^*_RK)$ &
$B_\ell$
&$Q\big(B^*_L- A^*_R K^{-1}\big)$\tsep{2pt}
\\
{\rm IIIS} &
$A_r$ &
$Q(A^*_R-B^*_L K)$ &
$B_r$
&
$Q\big(B^*_R-A^*_L K^{-1}\big)$
\\
{\rm IIIT} &
$B_\ell$
&$Q\big(B^*_L- A^*_R K^{-1}\big)$
&
$A_\ell$ &
$Q(A^*_L-B^*_RK)$
\\
{\rm IIIST} &
$B_r$
& $Q\big(B^*_R-A^*_L K^{-1}\big)$
&
$A_r$ &
$Q(A^*_R-B^*_L K)$
 \end{tabular}
}
\medskip

\centerline{
\begin{tabular}[t]{c|cccc}
{\rm module label}& $x_0$ & $x_1$ & $x_2$ & $x_3$
 \\ \hline
{\rm IV} &
$Q(A_\ell-K B_r)$
&
$A^*_L$
&
$Q\big(B_\ell-K^{-1} A_r\big)$ &
$B^*_L$
\\
{\rm IVS}
&
$Q(A_r-KB_\ell)$
&
$A^*_R$
&
$Q\big(B_r-K^{-1} A_\ell\big)$
&
$B^*_R$
\\
{\rm IVT} &
$Q\big(B_\ell-K^{-1} A_r\big)$ &
$B^*_L$ &
$Q(A_\ell-K B_r)$
&
$A^*_L$
\\
{\rm IVST}
&
$Q\big(B_r-K^{-1} A_\ell\big)$
&
$B^*_R$
&
$Q(A_r-KB_\ell)$
&
$A^*_R$
 \end{tabular}
}
\end{Theorem}

\begin{proof} By Proposition~\ref{prop:fourmod} and Lemma~\ref{lem:alwaysUmod}(ii).
\end{proof}

\begin{Note} Going forward, the $\square_q$-module $\mathbb V/J$ with label $\rm I$ will be denoted $(\mathbb V/J)_{\rm I}$, and so on.
\end{Note}

\begin{Proposition} \label{lem:Quotdder}For each $\square_q$-module in Theorem~{\rm \ref{prop:Fourmod2main}} the elements $x_1$ and $x_3$ act on the algebra $\mathbb V/J$ as a derivation of the following sort:

\centerline{
\begin{tabular}[t]{c|cc}
{\rm module label}& $x_1$ & $x_3$
 \\ \hline
{\rm I, II} &
{\rm $(K,I)$-derivation}
&
{\rm $\big(K^{-1},I\big)$-derivation}\tsep{2pt}
\\
{\rm IS, IIS} &
{\rm $(I,K)$-derivation}
&
{\rm $\big(I, K^{-1}\big)$-derivation}
\\
{\rm IT, IIT} &
{\rm $\big(K^{-1},I\big)$-derivation}
&
{\rm $( K,I)$-derivation}
\\
{\rm IST, IIST} &
{\rm $\big(I,K^{-1}\big)$-derivation}
&
{\rm $(I,K)$-derivation}
 \end{tabular}
}
\end{Proposition}
\begin{proof} By Lemma~\ref{lem:dder} and the comments about derivations in Section~\ref{section2}.
\end{proof}

\begin{Proposition} \label{lem:Quotdder2}For each $\square_q$-module in Theorem~{\rm \ref{prop:Fourmodmain}} the elements $x_1$ and $x_3$ act on the algebra $U$ as a~derivation of the following sort:

\centerline{
\begin{tabular}[t]{c|cc}
{\rm module label}& $x_1$ & $x_3$
 \\ \hline
{\rm III, IV} &
{\rm $(K,I)$-derivation}
&
{\rm $\big(K^{-1},I\big)$-derivation}\tsep{2pt}
\\
{\rm IIIS, IVS} &
{\rm $(I,K)$-derivation}
&
{\rm $\big(I, K^{-1}\big)$-derivation}
\\
{\rm IIIT, IVT} &
{\rm $\big(K^{-1},I\big)$-derivation}
&
{\rm $( K,I)$-derivation}
\\
{\rm IIIST, IVST} &
{\rm $\big(I,K^{-1}\big)$-derivation}
&
{\rm $(I,K)$-derivation}
 \end{tabular}
}
\end{Proposition}
\begin{proof} By Lemma~\ref{lem:dder2} and the comments about derivations in Section~\ref{section2}.
\end{proof}

\section[The sixteen $\square_q$-modules are mutually isomorphic and irreducible]{The sixteen $\boldsymbol{\square_q}$-modules are mutually isomorphic\\ and irreducible}\label{section14}

In the previous section we gave sixteen $\square_q$-module structures on $\mathbb V/J$ or $U$. In this section we show that these $\square_q$-modules are mutually isomorphic and irreducible.

\begin{Proposition}\label{prop:Smix} The map $\mathbb V/J\to \mathbb V/J$, $x+J\mapsto S(x)+J$ is an isomorphism of $\square_q$-modules from
\begin{alignat*}{3}
&(\mathbb V/J)_{\rm I} \leftrightarrow (\mathbb V/J)_{\rm IS}, \qquad &&
(\mathbb V/J)_{\rm IT} \leftrightarrow (\mathbb V/J)_{\rm IST},&\\
&(\mathbb V/J)_{\rm II} \leftrightarrow (\mathbb V/J)_{\rm IIS}, \qquad &&
(\mathbb V/J)_{\rm IIT} \leftrightarrow (\mathbb V/J)_{\rm IIST}.&
\end{alignat*}
The map $U \to U$, $x \mapsto S(x)$ is an isomorphism of $\square_q$-modules from
\begin{gather*}
U_{\rm III} \leftrightarrow U_{\rm IIIS}, \qquad
U_{\rm IIIT} \leftrightarrow U_{\rm IIIST}, \qquad
U_{\rm IV} \leftrightarrow U_{\rm IVS}, \qquad
U_{\rm IVT} \leftrightarrow U_{\rm IVST}.
\end{gather*}
\end{Proposition}
\begin{proof} By Lemma~\ref{lem:SIso}, together with the fact that $S(J)=J$ by~(\ref{eq:KSTJ}) and $S(U)=U$ by Lemma~\ref{lem:KSTU}.
\end{proof}

\begin{Proposition}\label{prop:Tmix}The map $\mathbb V/J\to \mathbb V/J$, $x+J\mapsto T(x)+J$ is an isomorphism of $\square_q$-modules from
\begin{alignat*}{3}
& (\mathbb V/J)_{\rm I} \leftrightarrow (\mathbb V/J)_{\rm IT}, \qquad &&
(\mathbb V/J)_{\rm IS} \leftrightarrow (\mathbb V/J)_{\rm IST},&\\
&(\mathbb V/J)_{\rm II} \leftrightarrow (\mathbb V/J)_{\rm IIT}, \qquad &&
(\mathbb V/J)_{\rm IIS} \leftrightarrow (\mathbb V/J)_{\rm IIST}.&
\end{alignat*}
The map $U \to U$, $x \mapsto T(x)$ is an isomorphism of $\square_q$-modules from
\begin{gather*}
 U_{\rm III} \leftrightarrow U_{\rm IIIT}, \qquad
U_{\rm IIIS} \leftrightarrow U_{\rm IIIST}, \qquad
U_{\rm IV} \leftrightarrow U_{\rm IVT}, \qquad
U_{\rm IVS} \leftrightarrow U_{\rm IVST}.
\end{gather*}
\end{Proposition}
\begin{proof} By Lemma~\ref{lem:TIso}, together with the fact that $T(J)=J$ by~(\ref{eq:KSTJ}) and $T(U)=U$ by Lemma~\ref{lem:KSTU}.
\end{proof}

\begin{Proposition}\label{prop:thetamix}The map $\mathbb V/J\to U$, $x+J\mapsto \theta(x)$ is an isomorphism of $\square_q$-modules from
\begin{alignat*}{5}
&(\mathbb V/J)_{\rm I} \rightarrow U_{\rm III}, \qquad &&
(\mathbb V/J)_{\rm IS} \rightarrow U_{\rm IIIS}, \qquad &&
(\mathbb V/J)_{\rm IT} \rightarrow U_{\rm IIIT}, \qquad &&
(\mathbb V/J)_{\rm IST} \rightarrow U_{\rm IIIST},&\\
&(\mathbb V/J)_{\rm II} \rightarrow U_{\rm IV}, \qquad &&
(\mathbb V/J)_{\rm IIS} \rightarrow U_{\rm IVS}, \qquad &&
(\mathbb V/J)_{\rm IIT} \rightarrow U_{\rm IVT}, \qquad &&
(\mathbb V/J)_{\rm IIST} \rightarrow U_{\rm IVST}.&
\end{alignat*}
\end{Proposition}
\begin{proof}By Lemma \ref{lem:thetamix} and the fact that $J$ is the kernel of~$\theta$.
\end{proof}

\begin{Proposition}\label{prop:varphimix}The map $\mathbb V/J\to \mathbb V/J$, $x+J\mapsto \varphi(x)+J$ is an isomorphsim of $\square_q$-modules from
\begin{gather*}
 (\mathbb V/J)_{\rm I} \rightarrow (\mathbb V/J)_{\rm II}, \qquad (\mathbb V/J)_{\rm IT} \rightarrow (\mathbb V/J)_{\rm IIT}.
\end{gather*}
The map $U\to U$, $x\mapsto \varphi^*(x)$
is an isomorphism of $\square_q$-modules from
\begin{gather*}
 U_{\rm III} \rightarrow U_{\rm IV}, \qquad U_{\rm IIIT} \rightarrow U_{\rm IVT}.
\end{gather*}
\end{Proposition}
\begin{proof} By Proposition~\ref{prop:expand} and Lemmas~\ref{lem:phiJU},~\ref{lem:twoincl}.
\end{proof}

\begin{Theorem}\label{thm:alliso}The following $\square_q$-modules are mutually isomorphic:
\begin{alignat*}{5}
& (\mathbb V/J)_{\rm I}, \qquad&&
(\mathbb V/J)_{\rm IS}, \qquad&&
(\mathbb V/J)_{\rm IT}, \qquad&&
(\mathbb V/J)_{\rm IST},&\\
& (\mathbb V/J)_{\rm II}, \qquad&&
(\mathbb V/J)_{\rm IIS}, \qquad&&
(\mathbb V/J)_{\rm IIT}, \qquad&&
(\mathbb V/J)_{\rm IIST},&\\
&U_{\rm III}, \qquad &&
U_{\rm IIIS}, \qquad &&
U_{\rm IIIT}, \qquad &&
U_{\rm IIIST},&\\
&U_{\rm IV}, \qquad &&
U_{\rm IVS}, \qquad&&
U_{\rm IVT}, \qquad &&
U_{\rm IVST}.&
\end{alignat*}
\end{Theorem}
\begin{proof} By Propositions \ref{prop:Smix}--\ref{prop:varphimix}.
\end{proof}

Our next goal is to show that the $\square_q$-modules in Theorem~\ref{thm:alliso} are irreducible.

\begin{Lemma} \label{lem:irred1}Let $W$ denote a nonzero $\square_q$-submodule of $U_{\rm IV}$. Then $1 \in W$.
\end{Lemma}
\begin{proof}By Proposition~\ref{prop:fourmod} the generators $x_1$ and $x_3$ act on $U_{\rm IV}$ as $A^*_L$ and $B^*_L$, respectively. Now $1 \in W$ by Lemma~\ref{lem:ABirred}.
\end{proof}

\begin{Lemma}\label{lem:irred2}Let $W$ denote a proper $\square_q$-submodule of $U_{\rm III}$. Then $1 \not\in W$.
\end{Lemma}
\begin{proof}We assume $1 \in W$ and get a contradiction. By Proposition~\ref{prop:fourmod} the genera\-tors~$x_0$ and~$x_2$ act on $U_{\rm III}$ as $A_\ell$ and $B_\ell$, respectively. By Lemma~\ref{lem:ABirred2}, $W$ is not proper. This contradicts our assumptions.
\end{proof}

\begin{Theorem}\label{thm:allirred}The $\square_q$-modules in Theorem~{\rm \ref{thm:alliso}} are irreducible.
\end{Theorem}
\begin{proof} By Theorem~\ref{thm:alliso} along with Lemmas~\ref{lem:irred1},~\ref{lem:irred2} we find that
none of the listed $\square_q$-modules contain a~nonzero proper $\square_q$-submodule. The result follows.
\end{proof}

\section[The NIL modules for $\square_q$]{The NIL modules for $\boldsymbol{\square_q}$}\label{section15}

In this section we characterize the $\square_q$-modules in Theorem~\ref{thm:alliso}, using the notion of a NIL $\square_q$-module.

We start with some comments about $\square^\vee_q$. Reformulating the relations~(\ref{eq:checkrels}),
\begin{alignat*}{3}
&x_1 x_0 = q^2 x_0 x_1 + 1-q^2, \qquad && x_1 x_2 = q^{-2} x_ 2x_1 + 1-q^{-2},&\\
&x_3 x_2 = q^2 x_2 x_3 + 1-q^2, \qquad && x_3 x_0 = q^{-2} x_0 x_3 + 1-q^{-2}.&
\end{alignat*}
Next we express these relations in a uniform way.

\begin{Lemma} \label{lem:unif}For $u \in \lbrace x_0, x_2\rbrace$ and $v \in \lbrace x_1, x_3\rbrace$ the following holds in $\square^\vee_q$:
\begin{gather*}
v u = uv q^{\langle u,v\rangle} + 1-q^{\langle u,v\rangle},
\end{gather*}
where
\[
\begin{array}{c|cc}
\langle\,,\,\rangle & x_1 & x_3 \\ \hline
x_0 &2 & -2 \\
x_2 & -2 & 2
 \end{array}
\]
\end{Lemma}

The following formula will be useful.
\begin{Lemma}\label{lem:formula}Pick $n \in \mathbb N$. Referring to the algebra $\square^\vee_q$, pick
$u_i\in\lbrace x_0, x_2\rbrace$ for $1 \leq i \leq n$, and also $v \in \lbrace x_1, x_3\rbrace $. Then
\begin{gather*}
v u_1u_2\cdots u_n = u_1 u_2 \cdots u_n v q^{ \langle u_1,v\rangle+\langle u_2,v\rangle+\cdots+
\langle u_n,v\rangle}\\
\hphantom{v u_1u_2\cdots u_n =}{} + \sum_{i=1}^n u_1\cdots u_{i-1} u_{i+1} \cdots u_n
q^{\langle u_1,v\rangle+\cdots + \langle u_{i-1},v\rangle}\big(1-q^{\langle u_i,v\rangle}\big),
\end{gather*}
where $\langle \,,\,\rangle$ is from Lemma~{\rm \ref{lem:unif}}.
\end{Lemma}
\begin{proof} By Lemma \ref{lem:unif} and induction on $n$.
\end{proof}

\begin{Corollary}\label{cor:modW}Let $V$ denote a~$\square^\vee_q$-module. Pick $\xi \in V$ such that $x_1 \xi = 0$ and $x_3 \xi = 0$. Then for $n \in \mathbb N$ and $u_1, u_2, \ldots, u_n \in\lbrace x_0, x_2\rbrace$ and $v \in \lbrace x_1, x_3\rbrace $,
\begin{gather*}
 v u_1u_2\cdots u_n \xi =\sum_{i=1}^n u_1\cdots u_{i-1} u_{i+1} \cdots u_n \xi
q^{\langle u_1,v\rangle+\cdots + \langle u_{i-1},v\rangle}\big(1-q^{\langle u_i,v\rangle}\big),
\end{gather*}
where $\langle \,,\,\rangle$ is from Lemma~{\rm \ref{lem:unif}}.
\end{Corollary}
\begin{proof} Referring to the equation displayed in Lemma~\ref{lem:formula}, apply each side to $\xi$ and note that $v\xi =0$.
\end{proof}

Recall that $\big(\square^\vee_q\big)^{\rm even}$ is the subalgebra of $\square^\vee_q$ generated by~$x_0$,~$x_2$. Recall the free algebra~$\mathbb V$ with generators~$A$,~$B$. By our comments below Note~\ref{note:square}, there exists an algebra isomorphism $\mathbb V\to \big(\square^\vee_q\big)^{\rm even}$ that sends $A \mapsto x_0$ and $B \mapsto x_2$. Denote this isomorphism by~$\kappa $. Recall the $\square^\vee_q $-module $\mathbb V_{\rm I}$ from Proposition~\ref{prop:fourmod2}. \begin{Lemma}\label{lem:Weta} Let $V$ denote a~$\square^\vee_q$-module. Pick $\xi \in V$ such that $x_1 \xi = 0$ and $x_3 \xi = 0$. Then the map $\mathbb V_I \to V$, $v \mapsto \kappa (v) \xi$ is a~$\square^\vee_q$-module homomorphism.
\end{Lemma}
\begin{proof} We show that for $i \in {\mathbb Z}_4 $ the following diagram commutes:
\begin{gather*}
\begin{CD}
\mathbb V_{\rm I} @>v \mapsto \kappa(v)\xi >>
 V
	 \\
 @Vx_i VV @VVx_i V \\
 \mathbb V_{\rm I} @>>v \mapsto \kappa(v) \xi> V.
 \end{CD}
\end{gather*}
Referring to the above diagram, the action $x_i\colon \mathbb V_{\rm I} \to\mathbb V_{\rm I}$
is described in Proposition~\ref{prop:fourmod2} along with Definition~\ref{def:Kdef} and Lemma~\ref{lem:dualaction}. The action $x_i\colon V\to V$ is clear for $i$ even, and described in
Corollary~\ref{cor:modW} for $i$ odd. Using these descriptions we chase each word in $\mathbb V$ around the diagram, and confirm that the diagram commutes.
\end{proof}

Let $V$ denote a $\square_q$-module. We view $V$ as a $\square^\vee_q$-module on which
\begin{gather*}
 x^3_i x_{i+2} - \lbrack 3 \rbrack_q x^2_i x_{i+2} x_i + \lbrack 3 \rbrack_q x_i x_{i+2} x^2_i -x_{i+2} x^3_i = 0, \qquad i \in \mathbb Z_4.
\end{gather*}

\begin{Lemma}\label{lem:qSerreW} Let $V$ denote a $\square_q$-module that contains a~nonzero vector $\xi$ such that $x_1 \xi = 0 $ and $x_3 \xi = 0$. Then the map in Lemma~{\rm \ref{lem:Weta}} has kernel~$J$.
Moreover the map $(\mathbb V/J)_{\rm I} \to V$, $v+J \mapsto \kappa(v)\xi$ is an injective $\square_q$-module homomorphism.
\end{Lemma}
\begin{proof}Let $L$ denote the kernel of the map in Lemma~\ref{lem:Weta}. This map sends $1 \mapsto \xi$, and $\xi$ is nonzero, so $1 \not \in L$. Observe that $L$ is a left ideal of the free algebra~$\mathbb V$.
The set $H= \lbrace v \in \mathbb V| \kappa(v)V=0\rbrace$ is a 2-sided ideal of the free algebra~$\mathbb V$. By construction $H \subseteq L$. Recall the elements $J^{\pm}$ from~(\ref{eq:jplus}),~(\ref{eq:jminus}). We have
\begin{gather}
\kappa(J^+) =x^3_0 x_{2} - \lbrack 3 \rbrack_q x^2_0 x_{2} x_0 +\lbrack 3 \rbrack_q x_0 x_{2} x^2_0 -x_{2} x^3_0,\label{eq:k1}\\
\kappa(J^-)=x^3_2 x_{0} - \lbrack 3 \rbrack_q x^2_2 x_{0} x_2 +\lbrack 3 \rbrack_q x_2 x_{0} x^2_2 -x_{0} x^3_2.\label{eq:k2}
\end{gather}
Since $V$ is a $\square_q$-module, the elements~(\ref{eq:k1}),~(\ref{eq:k2}) are zero on~$V$. Therefore
$J^\pm \in H$. Recall that $J$ is the 2-sided ideal of the free algebra $\mathbb V$ generated by $J^{\pm}$. Consequently $J \subseteq H$. We mentioned earlier that $H \subseteq L$, so $J \subseteq L$.
By this and Lemma~\ref{lem:Weta}, the map $(\mathbb V/J)_{\rm I} \to V$, $v+J \mapsto \kappa(v)\xi$
is a $\square_q$-module homomorphism that has kernel~$L/J$. This kernel $L/J$ is a $\square_q$-submodule of the $\square_q$-module $(\mathbb V/J)_{\rm I}$, and the $\square_q$-module $(\mathbb V/J)_{\rm I}$ is irreducible by Theorem~\ref{thm:allirred}, so $L/J=0$ or $L/J=\mathbb V/J$. Thus $L=J$ or $L=\mathbb V$.
We have $L\not=\mathbb V$ since $1 \not\in L$, so $L=J$. Consequently the map $(\mathbb V/J)_{\rm I} \to V$, $v+J \mapsto \kappa(v)\xi$ is injective. The result follows.
\end{proof}

Let $V$ denote a nonzero $\square_q$-module. For $\xi\in V$, we say that $V$ is {\it generated by $\xi$} whenever $V$ does not have a proper $\square_q$-submodule that contains $\xi$.

\begin{Proposition}\label{prop:bij} Let $V$ denote a $\square_q$-module that is generated by a nonzero vector $\xi$ such that $x_1 \xi = 0 $ and $x_3 \xi = 0$. Then the map $(\mathbb V/J)_{\rm I} \to V$, $v+J \mapsto \kappa(v)\xi$ is an isomorphism of $\square_q$-modules.
\end{Proposition}
\begin{proof}By Lemma~\ref{lem:qSerreW}, the map $(\mathbb V/J)_{\rm I} \to V$, $v+J \mapsto \kappa(v)\xi$
is an injective $\square_q$-module homomorphism. For this map the image is a~$\square_q$-submodule of $V$ that contains $\xi$, so this image is equal to~$V$. By these comments the map $(\mathbb V/J)_{\rm I} \to V$, $v+J \mapsto \kappa(v)\xi$ is an isomorphism of $\square_q$-modules.
\end{proof}

\begin{Theorem}\label{thm:Uchar} For a $\square_q$-module $V$ the following are equivalent:
\begin{enumerate}\itemsep=0pt
\item[\rm (i)] $V$ is isomorphic to the $\square_q$-modules in Theorem~{\rm \ref{thm:alliso}};
\item[\rm (ii)] $v$ is generated by a nonzero vector $\xi$ such that $x_1 \xi = 0 $ and $x_3 \xi = 0$.
\end{enumerate}
\end{Theorem}
\begin{proof}${\rm (i)} \Rightarrow {\rm (ii)}$: Without loss of generality, we may identify the $\square_q$-module $V$ with the $\square_q$-module $U_{\rm III}$ listed in Theorem~\ref{thm:alliso}.
The vector $\xi=1$ has the desired properties.

${\rm (ii)} \Rightarrow {\rm (i)}$: By Theorem~\ref{thm:alliso} and Proposition~\ref{prop:bij}.
\end{proof}

\begin{Definition}Let $V$ denote a $\square_q$-module. A vector $\xi \in V$ is called {\rm NIL} whenever
$x_1 \xi = 0$ and $x_3 \xi = 0$ and $\xi \not= 0$. The $\square_q$-module $V$ is called {\it NIL} whenever it is generated by a NIL vector.
\end{Definition}

By Theorem~\ref{thm:Uchar}, up to isomorphism there exists a unique NIL $\square_q$-module, which we denote by $\bf U$. Also by Theorem~\ref{thm:Uchar}, the $\square_q$-module $\bf U$ is isomorphic to each of the
$\square_q$-modules from Theorem~\ref{thm:alliso}. By Theorem~\ref{thm:allirred} the $\square_q$-module $\bf U$ is irreducible. The $\square_q$-module $\bf U$ is infinite-dimensional; indeed it is isomorphic to $U^+_q$ as a vector space, as we now clarify. Recall the algebra isomorphism $U^+_q \to \square^{\rm even}_q$ from below Lemma~\ref{lem:add}.

\begin{Lemma}Identify the algebra $U^+_q$ with $\square^{\rm even}_q$ via the algebra isomorphism from
below Lemma~{\rm \ref{lem:add}}. Let $\xi $ denote a NIL vector in~${\bf U}$. Then the map
$U^+_q \to {\bf U}$, $u \mapsto u \xi$ is an isomorphism of vector spaces.
\end{Lemma}
\begin{proof} By Proposition \ref{prop:bij}.
\end{proof}

\begin{Theorem}The $\square_q$-module ${\bf U}$ has a unique sequence of subspaces
$\lbrace {\bf U}_n\rbrace_{n \in \mathbb N}$ such that:
\begin{enumerate}\itemsep=0pt
\item[\rm (i)] ${\bf U}_0 \not=0$;
\item[\rm (ii)] the sum ${\bf U} = \sum\limits_{n \in \mathbb N} {\bf U}_n$ is direct;
\item[\rm (iii)] for $n \in \mathbb N$,
\begin{gather*}
x_0 {\bf U}_n \subseteq {\bf U}_{n+1}, \qquad
x_1 {\bf U}_n \subseteq {\bf U}_{n-1}, \qquad
x_2 {\bf U}_n \subseteq {\bf U}_{n+1}, \qquad
x_3 {\bf U}_n \subseteq {\bf U}_{n-1},
\end{gather*}
where ${\bf U}_{-1}=0$.
\end{enumerate}
The sequence $\lbrace {\bf U}_n\rbrace_{n \in \mathbb N}$ is described as follows. The subspace
 ${\bf U}_0$ has dimension~$1$. The nonzero vectors in ${\bf U}_0$ are precisely the NIL vectors
in ${\bf U}$, and each of these vectors
generates ${\bf U}$.
Let $\xi$ denote a NIL vector in $\bf U$.
Then for $n \in \mathbb N$,
 ${\bf U}_n$ is spanned by
the vectors
\begin{gather}\label{eq:vectorlist}
u_1u_2\cdots u_n \xi,\qquad u_i \in \lbrace x_0, x_2\rbrace, \qquad 1 \leq i \leq n.
\end{gather}
\end{Theorem}
\begin{proof}Concerning existence, without loss of generality we may identify the $\square_q$-module ${\bf U}$ with the $\square_q$-module $U_{\rm III}$ listed in Theorem~\ref{thm:alliso}. Recall that $U$ is the subalgebra of the $q$-shuffle algebra $\mathbb V$ generated by~$A$,~$B$. For $n \in \mathbb N$ define ${\bf U}_n = U\cap \mathbb V_n$. One checks that the sequence $\lbrace {\bf U}_n\rbrace_{n \in \mathbb N}$ satisfies the above conditions (i)--(iii). We have established existence. Going forward let $\lbrace {\bf U}_n\rbrace_{n \in \mathbb N}$ denote any sequence of subspaces that satisfies the above conditions (i)--(iii). Let~$\xi$ denote a NIL vector in ${\bf U}$. We claim that $\xi \in {\bf U}_0$. To prove the claim, note by condition~(ii) that there exists $n \in \mathbb N$ and $\xi_i \in {\bf U}_i$ $(0 \leq i \leq n)$ such that $\xi_n\not=0$ and $\xi = \sum\limits_{i=0}^n \xi_i$. Since $\xi $ is NIL, $x_1 \xi = 0 $ and
$x_3 \xi=0$. In the sum $\xi = \sum\limits_{i=0}^n \xi_i$, apply $x_1$ to each term and use condition~(iii) to find $x_1 \xi_n = 0$. Similarly $x_3 \xi_n=0$, so $\xi_n$ is NIL. Since the $\square_q$-module ${\bf U}$ is irreducible, it is generated by any nonzero vector in ${\bf U}$. In particular the $\square_q$-module ${\bf U}$ is generated by $\xi_n$. By Proposition~\ref{prop:bij} the map $(\mathbb V/J)_{\rm I} \to {\bf U}$, $v + J \mapsto \kappa(v)\xi_n$ is an isomorphism of $\square_q$-modules. Consider the image of this map. On one hand, the image is equal to $\bf U$. On the other hand, by~(iii) and the definition of $\kappa$ above Lemma~\ref{lem:Weta}, the image is contained in $\sum\limits_{i\in \mathbb N} {\bf U}_{n+i}$.
By these comments and~(i), (ii) we obtain $n=0$, so $\xi =\xi_0 \in {\bf U}_0$. We have proven the claim.
For $n \in \mathbb N$ let ${\bf U}'_n$ denote the subspace of $\bf U$ spanned by the vectors~(\ref{eq:vectorlist}). We claim that ${\bf U}'_n= {\bf U}_n$ for $n \in \mathbb N$. By~(iii) we have
${\bf U}'_n \subseteq {\bf U}_n$ for $n \in \mathbb N$. Earlier we mentioned an isomorphism $(\mathbb V/J)_{\rm I} \to {\bf U}$; its existence shows that the vector space ${\bf U}$ is spanned by the vectors
\begin{gather*}
u_1 u_2 \cdots u_n \xi, \qquad n \in \mathbb N, \qquad
u_i \in \lbrace x_0, x_2\rbrace, \qquad 1 \leq i \leq n.
\end{gather*}
So ${\bf U} = \sum\limits_{n \in \mathbb N} {\bf U}'_n$. By these comments and (ii) we obtain ${\bf U}'_n = {\bf U}_n$ for $n \in \mathbb N$. The claim is proven. By the claims, the sequence $\lbrace {\bf U}_n\rbrace_{n \in \mathbb N}$ satisfying conditions (i)--(iii) is unique, and it fits the description given in the last paragraph of the theorem statement.
\end{proof}

\appendix

\section[Data on the $q$-shuffle product]{Data on the $\boldsymbol{q}$-shuffle product}\label{appendixA}

Recall the $q$-shuffle algebra $\mathbb V$ from Section~\ref{section6}. In the following tables we express some $q$-shuffle products in terms of the standard basis for~$\mathbb V$.

\centerline{
\begin{tabular}[t]{c|ccc}
 & $A\star A \star B$ & $A \star B \star A$ & $B \star A \star A$
 \\ \hline
$AAB$ &
$q \lbrack 2 \rbrack_q $ & $q^{-1} \lbrack 2 \rbrack_q$ & $q^{-3} \lbrack 2 \rbrack_q$\tsep{2pt}
 \\
$ABA$ &
 $q^{-1} \lbrack 2 \rbrack_q$ &
 $q^{-1} \lbrack 2 \rbrack_q$ &
 $q^{-1} \lbrack 2 \rbrack_q$
\\
$BAA$ & $q^{-3}\lbrack 2 \rbrack_q $
&
$q^{-1}\lbrack 2 \rbrack_q$
&
$q\lbrack 2 \rbrack_q$
\\
 \end{tabular}
}

\medskip

\centerline{
\begin{tabular}[t]{c|ccc}
 & $B\star B \star A$ & $B \star A \star B$ & $A \star B \star B$
 \\ \hline
$BBA$ &
$q \lbrack 2 \rbrack_q $ & $q^{-1} \lbrack 2 \rbrack_q$ & $q^{-3} \lbrack 2 \rbrack_q$\tsep{2pt}
 \\
$BAB$ &
 $q^{-1} \lbrack 2 \rbrack_q$ &
 $q^{-1} \lbrack 2 \rbrack_q$ &
 $q^{-1} \lbrack 2 \rbrack_q$
\\
$ABB$ & $q^{-3}\lbrack 2 \rbrack_q $
&
$q^{-1}\lbrack 2 \rbrack_q$
&
$q\lbrack 2 \rbrack_q$
\\
 \end{tabular}
}
\medskip

\centerline{
\begin{tabular}[t]{c|cccc}
 & $A\star A \star A \star B$ & $A \star A \star B \star A$ & $A \star B \star A \star A$ & $B \star A \star A \star A$
 \\ \hline
$AAAB$ &
$q^3 \lbrack 3 \rbrack_q \lbrack 2 \rbrack_q$
&
$q \lbrack 3 \rbrack_q \lbrack 2 \rbrack_q $
&
$q^{-1} \lbrack 3 \rbrack_q \lbrack 2 \rbrack_q$
&
$q^{-3} \lbrack 3 \rbrack_q \lbrack 2 \rbrack_q$\tsep{2pt}
\\
$AABA$ & $q \lbrack 3 \rbrack_q \lbrack 2 \rbrack_q$
&
$\big(2q+q^{-1}\big)\lbrack 2 \rbrack_q$ &
$\big(q+ 2q^{-1}\big)\lbrack 2 \rbrack_q$
&
$q^{-1} \lbrack 3 \rbrack_q \lbrack 2 \rbrack_q $
\\
$ABAA$ &
$q^{-1}\lbrack 3 \rbrack_q \lbrack 2 \rbrack_q$
&
$\big(q+ 2q^{-1}\big)\lbrack 2 \rbrack_q$
&
$\big(2q+q^{-1}\big)\lbrack 2 \rbrack_q$ &
$q\lbrack 3 \rbrack_q \lbrack 2 \rbrack_q$
\\
$BAAA$ &
$q^{-3} \lbrack 3 \rbrack_q \lbrack 2 \rbrack_q$
&
$q^{-1} \lbrack 3 \rbrack_q \lbrack 2 \rbrack_q$
&
$q\lbrack 3 \rbrack_q \lbrack 2 \rbrack_q$
&
$q^3 \lbrack 3 \rbrack_q \lbrack 2 \rbrack_q$
 \end{tabular}
}
\medskip

\centerline{
\begin{tabular}[t]{c|cccc}
 & $B\star B \star B \star A$ & $B \star B \star A \star B$ & $B \star A \star B \star B$ & $A \star B \star B \star B$
 \\ \hline
$BBBA$ &
$q^3 \lbrack 3 \rbrack_q \lbrack 2 \rbrack_q$
&
$q \lbrack 3 \rbrack_q \lbrack 2 \rbrack_q $
&
$q^{-1} \lbrack 3 \rbrack_q \lbrack 2 \rbrack_q$
&
$q^{-3} \lbrack 3 \rbrack_q \lbrack 2 \rbrack_q$\tsep{2pt}
\\
$BBAB$ & $q \lbrack 3 \rbrack_q \lbrack 2 \rbrack_q$
&
$\big(2q+q^{-1}\big)\lbrack 2 \rbrack_q$ &
$\big(q+ 2q^{-1}\big)\lbrack 2 \rbrack_q$
&
$q^{-1} \lbrack 3 \rbrack_q \lbrack 2 \rbrack_q $
\\
$BABB$ &
$q^{-1}\lbrack 3 \rbrack_q \lbrack 2 \rbrack_q$
&
$\big(q+ 2q^{-1}\big)\lbrack 2 \rbrack_q$
&
$\big(2q+q^{-1}\big)\lbrack 2 \rbrack_q$ &
$q\lbrack 3 \rbrack_q \lbrack 2 \rbrack_q$
\\
$ABBB$ &
$q^{-3} \lbrack 3 \rbrack_q \lbrack 2 \rbrack_q$
&
$q^{-1} \lbrack 3 \rbrack_q \lbrack 2 \rbrack_q$
&
$q\lbrack 3 \rbrack_q \lbrack 2 \rbrack_q$
&
$q^3 \lbrack 3 \rbrack_q \lbrack 2 \rbrack_q$
 \end{tabular}
}

\section{Some matrix representations}\label{appendixB}

Consider the free algebra $\mathbb V$ generated by $A$, $B$. Earlier in the paper we described many maps
in ${\rm End}(\mathbb V)$. For such a map $X$, consider the matrix that represents~$X$ with respect to the standard basis for~$\mathbb V$. The rows and columns are indexed by the words in~$\mathbb V$. For words $u$, $v$ the $(u,v)$-entry is equal to $(u,Xv)$. We will display some of these entries shortly. For the above matrix and $r,s \in \mathbb N$ the submatrix $X(r,s)$ has rows and columns indexed by the words of length $r$ and $s$, respectively. The matrix $X(r,s)$ has dimensions $2^r\times 2^s$. We are going to display the nonzero~$X(r,s)$ such that $0\leq r,s\leq 3$. For this display we use the following word order:

\centerline{
\begin{tabular}[t]{c|c}
{\rm word length} & {\rm word order}
 \\ \hline
 $0$ & $1$
 \\
 $1$ &
 $A$, $B$
 \\
	$2$ &
	$AA$, $AB$, $BA$, $BB$
	\\
	$3$ &
	$AAA$,
	 $AAB$,
	 $ABA$,
	 $BAA$,
	 $ABB$,
	 $BAB$,
	 $BBA$,
	 $BBB$
	 \end{tabular}
	 }

For each $X$ we now display the nonzero $X(r,s)$ such that $0 \leq r,s\leq 3$. From the dimensions of $X(r,s)$ it is clear what is $r$ and $s$, so we do not state this explicitly. We have
\begin{gather*}
A_L\colon \
 \left(
 \begin{matrix}
 1 \\
 0
 \end{matrix}
 \right),
\qquad
 \left(
 \begin{matrix}
 1 & 0 \\
 0 & 1 \\
 0 & 0 \\
 0 & 0
 \end{matrix}
 \right),
\qquad
 \left(
 \begin{matrix}
 1 & 0 & 0 & 0 \\
 0 & 1 & 0 & 0 \\
 0 & 0 & 1 & 0 \\
 0 & 0 & 0 & 0 \\
 0 & 0 & 0 & 1 \\
 0 & 0 & 0 & 0 \\
 0 & 0 & 0 & 0 \\
 0 & 0 & 0 & 0
 \end{matrix}
 \right),
\\
A^*_L\colon \
 \left(
 \begin{matrix}
 1 & 0
 \end{matrix}
 \right),
 \qquad
	\left(
 \begin{matrix}
 1 & 0 & 0 & 0 \\
 0 & 1 & 0 & 0
 \end{matrix}
 \right),
\qquad
 \left(
 \begin{matrix}
 1 & 0 & 0 & 0 & 0 &0 &0 &0 \\
 0 & 1 & 0 & 0 & 0 & 0 &0 & 0 \\
 0 & 0 & 1 & 0 & 0 &0 & 0 & 0 \\
 0 & 0 & 0 & 0 & 1 & 0 &0 &0
 \end{matrix}
 \right),
\\
B_L\colon \
 \left(
 \begin{matrix}
 0 \\
 1
 \end{matrix}
 \right),
 \qquad
 \left(
 \begin{matrix}
 0 & 0 \\
 0 & 0 \\
 1 & 0 \\
 0 & 1
 \end{matrix}
 \right),
\qquad
 \left(
 \begin{matrix}
 0 & 0 & 0 & 0 \\
 0 & 0 & 0 & 0 \\
 0 & 0 & 0 & 0 \\
 1 & 0 & 0 & 0 \\
 0 & 0 & 0 & 0 \\
 0 & 1 & 0 & 0 \\
 0 & 0 & 1 & 0 \\
 0 & 0 & 0 & 1
 \end{matrix}
 \right),
\\
B^*_L\colon \
	\left(
 \begin{matrix}
 0 & 1
 \end{matrix}
 \right),
\qquad 	
	\left(
 \begin{matrix}
 0 & 0 & 1 & 0 \\
 0 & 0 & 0 & 1
 \end{matrix}
 \right),
\qquad 	
	\left(
 \begin{matrix}
 0 & 0 & 0 & 1 & 0 & 0 & 0 & 0 \\
 0 & 0 & 0 & 0 & 0 & 1 & 0 & 0 \\
 0 & 0 & 0 & 0 & 0 & 0 & 1 & 0 \\
 0 & 0 & 0 & 0 & 0 & 0 & 0 & 1
 \end{matrix}
 \right),
\\
A_R\colon \
 \left(
 \begin{matrix}
 1 \\
 0
 \end{matrix}
\right),
\qquad
 \left(
 \begin{matrix}
 1 & 0 \\
 0 & 0 \\
 0 & 1 \\
 0 & 0
 \end{matrix}
 \right),
\qquad
 \left(
 \begin{matrix}
 1 & 0 & 0 & 0 \\
 0 & 0 & 0 & 0 \\
 0 & 1 & 0 & 0 \\
 0 & 0 & 1 & 0 \\
 0 & 0 & 0 & 0 \\
 0 & 0 & 0 & 0 \\
 0 & 0 & 0 & 1 \\
 0 & 0 & 0 & 0
 \end{matrix}
 \right),
\\
A^*_R\colon \
 \left(
 \begin{matrix}
 1 & 0
 \end{matrix}
 \right),
\qquad
 \left(
 \begin{matrix}
 1 & 0 & 0 & 0 \\
 0 & 0 & 1 & 0
 \end{matrix}
 \right),
\qquad
 \left(
 \begin{matrix}
 1 & 0 & 0 & 0 &0&0&0&0 \\
 0 & 0 & 1 & 0 &0&0&0&0 \\
 0 & 0 & 0 & 1 &0&0&0&0 \\
 0 & 0 & 0 & 0 &0&0&1&0
 \end{matrix}
 \right),
\\
B_R\colon \
 \left(
 \begin{matrix}
 0 \\
 1
 \end{matrix}
 \right),
\qquad
 \left(
 \begin{matrix}
 0 & 0 \\
 1 & 0 \\
 0 & 0 \\
 0 & 1
 \end{matrix}
 \right),
\qquad
 \left(
 \begin{matrix}
 0 & 0 & 0 & 0 \\
 1 & 0 & 0 & 0 \\
 0 & 0 & 0 & 0 \\
 0 & 0 & 0 & 0 \\
 0 & 1 & 0 & 0 \\
 0 & 0 & 1 & 0 \\
 0 & 0 & 0 & 0 \\
 0 & 0 & 0 & 1
 \end{matrix}
 \right),
\\
B^*_R\colon \
 \left(
 \begin{matrix}
 0 & 1
 \end{matrix}
 \right),
 \qquad
 \left(
 \begin{matrix}
 0 & 1 & 0 & 0\\
 0 & 0 & 0 & 1
 \end{matrix}
 \right),
\qquad
 \left(
 \begin{matrix}
 0 & 1 & 0 & 0 &0&0&0&0 \\
 0 & 0 & 0 & 0 &1&0&0&0 \\
 0 & 0 & 0 & 0 &0&1&0&0 \\
 0 & 0 & 0 & 0 &0&0&0&1
 \end{matrix}
 \right),
\\
A_\ell\colon \
 \left(
 \begin{matrix}
 1 \\
 0
 \end{matrix}
 \right),
\qquad
 \left(
 \begin{matrix}
 q\lbrack 2 \rbrack_q & 0 \\
 0 & 1 \\
 0 & q^{-2} \\
 0 & 0
 \end{matrix}
 \right),
 \qquad
	\left(
 \begin{matrix}
 q^2\lbrack 3 \rbrack_q & 0 & 0 & 0 \\
 0 & q\lbrack 2 \rbrack_q & 0 & 0 \\
 0 & 1 & 1 & 0 \\
 0 & 0 & q^{-1}\lbrack 2 \rbrack_q & 0 \\
 0 & 0 & 0 & 1 \\
 0 & 0 & 0 & q^{-2} \\
 0 & 0 & 0 & q^{-4} \\
 0 & 0 & 0 & 0
 \end{matrix}
 \right),
\\
A^*_\ell\colon \
 \left(
 \begin{matrix}
 1 & 0
 \end{matrix}
 \right),
\qquad
 \left(
 \begin{matrix}
 q\lbrack 2 \rbrack_q & 0 &0& 0 \\
 0 & 1 & q^{-2} & 0
 \end{matrix}
 \right),
 \qquad
 \left(
 \begin{matrix}
 q^2\lbrack 3 \rbrack_q\!&0&0&0 &0&0&0&0 \\
 0&\!q\lbrack 2 \rbrack_q\!&1&0 &0&0&0&0 \\
 0&0&1&\!q^{-1}\lbrack 2 \rbrack_q\! &0&0&0&0 \\
 0&0&0 &0&1&\!q^{-2}\!& \!q^{-4}\! &0
		 \end{matrix}
 \right),
\\
B_\ell\colon \
	\left(
 \begin{matrix}
 0 \\
 1
 \end{matrix}
 \right),
\qquad 	
	\left(
 \begin{matrix}
 0 & 0 \\
 q^{-2} & 0 \\
 1 & 0 \\
 0 & q\lbrack 2 \rbrack_q
 \end{matrix}
 \right),
\qquad 	
	\left(
 \begin{matrix}
 0 & 0 & 0 & 0 \\
 q^{-4} & 0 & 0 & 0 \\
 q^{-2} & 0 & 0 & 0 \\
 1 & 0 & 0 & 0 \\
 0 & q^{-1}\lbrack 2 \rbrack_q & 0 & 0 \\
 0 & 1 & 1 & 0 \\
 0 & 0 & q\lbrack 2 \rbrack_q & 0 \\
 0 & 0 & 0 & q^2\lbrack 3 \rbrack_q
 \end{matrix}
 \right),
\\
B^*_\ell\colon \
 \left(
 \begin{matrix}
 0&1
		 \end{matrix}
 \right),
\qquad
 \left(
 \begin{matrix}
 0&q^{-2}&1&0 \\
 0&0&0&q\lbrack 2 \rbrack_q
		 \end{matrix}
 \right),
 \qquad
 \left(
 \begin{matrix}
 0&\!q^{-4}\!&\!q^{-2}\!&1&0&0&0&0 \\
 0&0&0&0&\!q^{-1}\lbrack 2 \rbrack_q\!&1&0&0 \\
 0&0&0&0&0&1&\!q\lbrack 2 \rbrack_q\!&0 \\
 0&0&0&0&0&0&0&\!q^2\lbrack 3 \rbrack_q
		 \end{matrix}
 \right),
\\
A_r\colon \
 \left(
 \begin{matrix}
 1 \\
 0
 \end{matrix}
\right),
\qquad
 \left(
 \begin{matrix}
 q\lbrack 2 \rbrack_q & 0
	 \\
 0 & q^{-2} \\
 0 & 1 \\
 0 & 0
 \end{matrix}
 \right),
\qquad
 \left(
 \begin{matrix}
 q^2\lbrack 3 \rbrack_q & 0 & 0 & 0 \\
 0 & q^{-1}\lbrack 2 \rbrack_q & 0 & 0 \\
 0 & 1 & 1 & 0 \\
 0 & 0 & q\lbrack 2 \rbrack_q & 0 \\
 0 & 0 & 0 & q^{-4} \\
 0 & 0 & 0 & q^{-2} \\
 0 & 0 & 0 & 1 \\
 0 & 0 & 0 & 0
 \end{matrix}
 \right),
\\
A^*_r\colon \
 \left(
 \begin{matrix}
 1&0
		 \end{matrix}
 \right),
\qquad
 \left(
 \begin{matrix}
 q\lbrack 2 \rbrack_q&0&0&0 \\
 0&q^{-2}& 1& 0
		 \end{matrix}
 \right),
 \qquad
 \left(
 \begin{matrix}
 q^2\lbrack 3 \rbrack_q\!&&0&0&0&0&0&0 \\
 0&\!q^{-1}\lbrack 2 \rbrack_q\!&1&0&0&0&0&0 \\
 0&0&1&\!q\lbrack 2\rbrack_q\! &0&0&0&0 \\
 0&0&0&0&\! q^{-4}&q^{-2}& 1& 0
		 \end{matrix}
 \right),
\\
B_r\colon \
 \left(
 \begin{matrix}
 0 \\
 1
 \end{matrix}
 \right),
\qquad
 \left(
 \begin{matrix}
 0 & 0 \\
 1 & 0 \\
 q^{-2} & 0 \\
 0 & q\lbrack 2 \rbrack_q
 \end{matrix}
 \right),
\qquad
 \left(
 \begin{matrix}
 0 & 0 & 0 & 0 \\
 1 & 0 & 0 & 0 \\
 q^{-2} & 0 & 0 & 0 \\
 q^{-4} & 0 & 0 & 0 \\
 0 & q\lbrack 2 \rbrack_q & 0 & 0 \\
 0 & 1 & 1 & 0 \\
 0 & 0 & q^{-1}\lbrack 2 \rbrack_q & 0 \\
 0 & 0 & 0 & q^2\lbrack 3 \rbrack_q
 \end{matrix}
 \right),
\\
B^*_r\colon \
 \left(
 \begin{matrix}
 0&1
		 \end{matrix}
 \right),
\qquad
 \left(
 \begin{matrix}
 0&1&q^{-2}&0 \\
 0&0&0&q\lbrack 2 \rbrack_q
		 \end{matrix}
 \right),
 \qquad
 \left(
 \begin{matrix}
 0&1&\!q^{-2}\!&\!q^{-4}\!&0&0&0&0 \\
 0&0&0&0&\!q\lbrack 2 \rbrack_q\!& 1&0&0 \\
 0&0&0&0&0&1&\!q^{-1}\lbrack 2 \rbrack_q\!&0 \\
 0&0&0&0&0&0&0&\!q^2\lbrack 3 \rbrack_q
		 \end{matrix}
 \right),
\\
K\colon \
 \left(
 \begin{matrix}
		 1
		 \end{matrix}
 \right), \qquad
 \left(
 \begin{matrix}
		 q^2&0 \\
		 0&q^{-2}
		 \end{matrix}
 \right),
\qquad
 \left(
 \begin{matrix}
		 q^4&0&0&0 \\
		 0&1&0&0 \\
		 0&0&1&0 \\
		 0&0&0&q^{-4}
		 \end{matrix}
 \right),
\\ \qquad
 \left(
 \begin{matrix}
		 q^6&0&0&0&0&0&0&0 \\
		 0&q^2&0&0&0&0&0&0 \\
		 0&0&q^2&0&0&0&0&0 \\
		 0&0&0&q^2&0&0&0&0 \\
		 0&0&0&0&q^{-2}&0&0&0 \\
		 0&0&0&0&0&q^{-2}&0&0 \\
		 0&0&0&0&0&0&q^{-2}&0 \\
		 0&0&0&0&0&0&0&q^{-6}
		 \end{matrix}
 \right),
\\
S\colon \
 \left(
 \begin{matrix}
		 1
		 \end{matrix}
 \right),
 \qquad
 \left(
 \begin{matrix}
		 1&0 \\
		 0&1
		 \end{matrix}
 \right),
 \qquad
 \left(
 \begin{matrix}
		 1&0&0&0 \\
		 0&0&1&0 \\
		 0&1&0&0 \\
		 0&0&0&1
		 \end{matrix}
 \right),
\qquad
 \left(
 \begin{matrix}
		 1&0&0&0&0&0&0&0 \\
		 0&0&0&1&0&0&0&0 \\
		 0&0&1&0&0&0&0&0 \\
		 0&1&0&0&0&0&0&0 \\
		 0&0&0&0&0&0&1&0 \\
		 0&0&0&0&0&1&0&0 \\
		 0&0&0&0&1&0&0&0 \\
		 0&0&0&0&0&0&0&1
		 \end{matrix}
 \right),
\\
T\colon \
\left(
 \begin{matrix}
		 1
		 \end{matrix}
\right),
\qquad
\left(
 \begin{matrix}
		 0&1 \\
		 1&0
		 \end{matrix}
\right),
\qquad
\left(
 \begin{matrix}
		 0&0&0&1 \\
		 0&0&1&0 \\
		 0&1&0&0 \\
		 1&0&0&0
		 \end{matrix}
 \right),
\qquad
 \left(
 \begin{matrix}
		 0&0&0&0&0&0&0&1 \\
		 0&0&0&0&0&0&1&0 \\
		 0&0&0&0&0&1&0&0 \\
		 0&0&0&0&1&0&0&0 \\
		 0&0&0&1&0&0&0&0 \\
		 0&0&1&0&0&0&0&0 \\
		 0&1&0&0&0&0&0&0 \\
		 1&0&0&0&0&0&0&0
		 \end{matrix}
 \right),
\\
 \theta\colon \
	\left(
 \begin{matrix}
		 1
		 \end{matrix}
 \right),
 \qquad
	\left(
 \begin{matrix}
		 1&0 \\
		 0&1
		 \end{matrix}
 \right),
\qquad
	\left(
 \begin{matrix}
		 q\lbrack 2 \rbrack_q&0&0&0 \\
		 0&1&q^{-2}&0 \\
		 0&q^{-2}&1&0 \\
		 0&0&0&q\lbrack 2 \rbrack_q
		 \end{matrix}
 \right),
\\
 \qquad
 \lbrack 2 \rbrack_q \left(
 \begin{matrix}
		 q^3\lbrack 3 \rbrack_q&0&0&0&0&0&0&0 \\
		 0&q&q^{-1}&q^{-3}&0&0&0&0 \\
		 0&q^{-1}&q^{-1}&q^{-1}&0&0&0&0 \\
		 0&q^{-3}&q^{-1}&q&0&0&0&0 \\
		 0&0&0&0&q&q^{-1}&q^{-3}&0 \\
		 0&0&0&0&q^{-1}&q^{-1}&q^{-1}&0 \\
		 0&0&0&0&q^{-3}&q^{-1}&q&0 \\
		 0&0&0&0&0&0&0&q^3 \lbrack 3 \rbrack_q
		 \end{matrix}
 \right),
\\
 \varphi \colon \
 \left(
 \begin{matrix}
		 1
		 \end{matrix}
 \right),
\qquad
 Q \left(
 \begin{matrix}
		 1&-q^{-2} \\
		 -q^{-2}& 1
		 \end{matrix}
 \right),
\qquad
 Q^2 \left(
 \begin{matrix}
		 1&-q^{-2}&-q^{-4}&q^{-6} \\
		 -q^{-1}\lbrack 2\rbrack_q&q^{-1}\lbrack 2 \rbrack_q&0&0 \\
		 0&0&q^{-1}\lbrack 2\rbrack_q&-q^{-1}\lbrack 2 \rbrack_q \\
		 q^{-6}&-q^{-4}&-q^{-2}&1
		 \end{matrix}
 \right),\\
Q^3 \left(\!
 \begin{matrix}
		 1\!&\!-q^{-2}\!&\!-q^{-4}\!&\!-q^{-6}\!&\!q^{-6}\!&\!q^{-8}\!&\!q^{-10}\!&\!-q^{-12} \\
		 -\lbrack 3 \rbrack_q\!&\!2+q^{-2}\!&\!q^{-2}\!&\!0\!&\!-q^{-4}\!&\!0\!&\!0\!&\!0 \\
		 0\!&\!0\!&\! q^{-1}\lbrack 2 \rbrack_q \!&\!q^{-3}\lbrack 2 \rbrack_q\!&\!\!-q^{-1}\lbrack 2 \rbrack_q\!&\!\!-q^{-3}\lbrack 2 \rbrack_q\!&\!0\!&\!0 \\
		 0\!&\!0\!&\!0\!&\!1\!&\!0\!&\! -q^{-2}\!&\!-q^{-2}-2q^{-4}\!&\!\!q^{-4}\lbrack 3 \rbrack_q \\
		 q^{-4}\lbrack 3 \rbrack_q\!&\!-q^{-2}-2q^{-4}\!&\!-q^{-2}\!&\!0\!&\!1\!&\!0\!&\!0\!&\!0 \\
		 0\!&\!0\!&\!-q^{-3}\lbrack 2 \rbrack_q \!&\!-q^{-1}\lbrack 2 \rbrack_q\!&\!
		 q^{-3}\lbrack 2 \rbrack_q\!&\!q^{-1}\lbrack 2 \rbrack_q\!&\!0\!&\!0 \\
		 0\!&\!0\!&\!0\!&\!-q^{-4}\!&\!0\!&\!q^{-2}\!&\!2+q^{-2}\!&\!-\lbrack 3 \rbrack_q \\
		 -q^{-12}\!&\!q^{-10}\!&\!q^{-8}\!&\!q^{-6}\!&\!-q^{-6}\!&\!-q^{-4}\!&\!-q^{-2}\!&\!1
		 \end{matrix}\!
 \right),
\\
 \varphi^* \colon \
 \left(
 \begin{matrix}
		 1
		 \end{matrix}
 \right),
 \qquad
 Q \left(
 \begin{matrix}
		 1&-q^{-2} \\
		 -q^{-2}& 1
		 \end{matrix}
 \right),
 \qquad
 Q^2 \left(
 \begin{matrix}
		 1&-q^{-1}\lbrack 2 \rbrack_q&0&q^{-6} \\
		 -q^{-2}&q^{-1}\lbrack 2 \rbrack_q&0&-q^{-4} \\
		-q^{-4}&0&q^{-1}\lbrack 2 \rbrack_q&-q^{-2} \\
		 q^{-6}&0&-q^{-1}\lbrack 2 \rbrack_q&1
		 \end{matrix}
 \right),
\\
Q^3
\left(\!
 \begin{matrix}
 1&\!-\lbrack 3 \rbrack_q\!&\!0\!&\!0\!&\!q^{-4} \lbrack 3 \rbrack_q\!&\!0\!&\!0\!&\!-q^{-12} \\
		 -q^{-2}\!&\!2+q^{-2}\!&\!0\!&\!0\!&\!-q^{-2}-2q^{-4}\!&\!0\!&\!0\!&\!q^{-10} \\
		 -q^{-4}\!&\!q^{-2}\!&\!q^{-1}\lbrack 2 \rbrack_q\!&\!0\!&\!-q^{-2}\!&\!-q^{-3}\lbrack 2 \rbrack_q\!&0&\!q^{-8} \\
		 -q^{-6}\!&\!0\!&\!q^{-3}\lbrack 2 \rbrack_q\!&\!1\!&\!0\!&\!-q^{-1}\lbrack 2 \rbrack_q\!&\!-q^{-4}\!&\!q^{-6} \\
		 q^{-6}\!&\!-q^{-4}\!&\!-q^{-1}\lbrack 2 \rbrack_q\!&\!0\!&\!1\!&\!q^{-3}\lbrack 2 \rbrack_q\! &\!0\!&\!-q^{-6} \\
		 q^{-8}\!&\!0\!&\!-q^{-3}\lbrack 2\rbrack_q\!&\!-q^{-2}&\!0\!&\!q^{-1}\lbrack 2\rbrack_q\! &\!q^{-2}\!&\!-q^{-4} \\
		 q^{-10}\!&\!0\!&\!0&\!-q^{-2}-2q^{-4}\!&\!0\!&\!0&\!2+q^{-2}\!&\!-q^{-2} \\
		 -q^{-12}\!&\!0\!&\!0\!&\!q^{-4} \lbrack 3 \rbrack_q\!&\!0\!&\!0\!&\!-\lbrack 3 \rbrack_q\!& \!1
		 \end{matrix}
 \right).
\end{gather*}

\subsection*{Acknowledgement}
The first author acknowledges support by the Simons Foundation Collaboration Grant 3192112. The second author thanks Marc Rosso and Xin Fang for helpful comments about $q$-shuffle algebras.

\pdfbookmark[1]{References}{ref}
\LastPageEnding

\end{document}